\documentclass[10pt,a4paper]{amsart}
\usepackage{hyperref}
\usepackage{mathrsfs}
\usepackage{amsmath}
\usepackage{amssymb}
\usepackage[all]{xy}
\usepackage{latexsym}
\usepackage[utf8]{inputenc}
\usepackage[T1]{fontenc}
\title[Essential extensions and the nilpotent filtration]{Essential extensions, the nilpotent filtration and the Arone-Goodwillie tower}
\author[G. Powell]{Geoffrey Powell}
\address{Laboratoire Angevin de Recherche en Mathématiques, UMR 6093, 
  Faculté des Sciences, Université d'Angers, 
2 Boulevard Lavoisier,
49045 Angers, France}
\email{Geoffrey.Powell@math.cnrs.fr}
\date{}
\newtheorem{THM}{Theorem}

\newtheorem{thm}{Theorem}[section]
\newtheorem{prop}[thm]{Proposition}
\newtheorem{cor}[thm]{Corollary}
\newtheorem{lem}[thm]{Lemma}
\theoremstyle{definition}
\newtheorem{defn}[thm]{Definition}
\newtheorem{exam}[thm]{Example}

\theoremstyle{remark}
\newtheorem{rem}[thm]{Remark}
\newtheorem{nota}[thm]{Notation}
\newtheorem{hyp}[thm]{Hypothesis}

\newcommand{\f}{\mathscr{F}}
\newcommand{\cale}{\mathscr{E}}
\newcommand{\nil}{\mathscr{N}il}
\newcommand{\ddl}{\mathfrak{Q}}
\newcommand{\symm}{\mathfrak{S}}
\renewcommand{\hom}{\mathrm{Hom}}
\newcommand{\ext}{\mathrm{Ext}}
\newcommand{\unstalg}{\mathscr{K}}
\newcommand{\unstalgaug}{\unstalg_a}
\newcommand{\dash}{\mbox{-}}
\newcommand{\largetensor}{\underline{\underline{\otimes}}}
\newcommand{\destab}{\Omega^\infty}

\newcommand{\amod}{\mathscr{M}}
\newcommand{\unst}{\mathscr{U}}
\renewcommand{\phi}{\varphi}
\renewcommand{\epsilon}{\varepsilon}
\newcommand{\nat}{\mathbb{N}}
\newcommand{\zed}{\mathbb{Z}}
\newcommand{\field}{\mathbb{F}}
\newcommand{\obj}{\mathrm{Ob}\ }
\newcommand{\op}{^\mathrm{op}}
\newcommand{\cala}{\mathscr{A}}

\newcommand{\id}{\mathrm{Id}}
\newcommand{\lc}{\mathfrak{C}}

\newcommand{\alunst}{\widetilde{\unst}}
\newcommand{\alnil}{\widetilde{\nil}}
\def\nlp{\mathrm{nil}}

\begin{document}

\begin{abstract}
The spectral sequence associated to the Arone-Goodwillie tower for the $n$-fold loop space functor 
is used to show that the first two non-trivial layers of the nilpotent filtration of the reduced mod $2$ cohomology of  a (sufficiently connected) space with 
nilpotent cohomology are comparable. This relies upon the theory of unstable modules over the mod $2$ Steenrod algebra, together 
with properties of a generalized class of almost unstable modules which is introduced here.

An essential ingredient of the proof is a non-vanishing result for certain extension groups in the category of unstable modules 
localized away from nilpotents.
\end{abstract}

\maketitle

\section{Introduction}
\label{sect:intro}

A fundamental question in algebraic topology is to ask what modules over the 
mod $p$ Steenrod algebra $\cala$ can be realized as 
the reduced mod $p$ cohomology $H^* (X)$ of a space. First obstructions are 
provided by the fact that the module must be unstable and that this structure
should
be compatible with the cup product. For example, Steenrod asked what polynomial 
algebras can be realized as the cohomology of a space. At the opposite extreme 
one 
can ask what unstable algebras with trivial cup square (at the prime two) can 
be 
realized;  a partial response is provided below in Theorem \ref{THM:K}.

The structure theory of the category $\unst$ of unstable modules allows the
formulation of precise questions of particular interest in the case where $H^* 
(X)$ is  nilpotent. 
 Kuhn \cite{Kuhn_annals} proposed a series of highly-influential 
non-realization 
conjectures, postulating significant restrictions on the structure of $H^* (X)$ 
as an unstable module.
 Many of these are now theorems \cite{Schwartz,CGPS}. The conjectures are 
phrased in terms of the nilpotent filtration of the category $\unst$; this is a 
decreasing filtration, where $\nil_n$
   is the smallest localizing subcategory of $\unst$ containing all $n$-fold 
suspensions. A general question is the following: what can be said about the 
structure of $H^* (X)$ as an unstable module 
   if it belongs to $\nil_n$?
   
   The condition $H^* (X) \in \nil_n$ has topological significance when Lannes' 
mapping space technology can be applied: it is equivalent
to $\mathrm{map} (BV, X)$ being 
$(n-1)$-connected for all elementary abelian $p$-groups $V$. It is clear that 
$n$-fold suspensions $\Sigma^n Y$ satisfy the hypothesis and, 
   since the algebraic suspension restricts to a functor $\Sigma: \nil_n 
\rightarrow\nil_{n+1}$, it is most interesting to consider the case where $X$
is 
not a suspension. 
   
   The largest submodule of an unstable module $M$ which lies in $\nil_n$ is 
written $\nlp_n M$; if $M$  lies in $\nil_n \backslash \nil_{n+1}$, then there 
is a short 
   exact sequence of unstable modules:
   \[
    0
    \rightarrow 
    \nlp_{n+1} M
    \rightarrow 
    M
    \rightarrow 
    \Sigma^n \rho_n M
    \rightarrow 
    0
   \]
where $\rho_n M$ is a reduced unstable module (ie contains no non-trivial
nilpotent 
submodule) which is non-zero. If in addition $M$ is $n$-connected, then $\rho_n 
M$ is connected (trivial in degree zero).

It is now known that (at least up to nilpotent unstable modules) there are few 
restrictions which can be placed on $\rho_n M$ (for $n \geq 1$), following the 
affirmation of the Lannes and Schwartz 
Artinian conjecture \cite{2014arXiv1408.3694P,2014arXiv1409.1670S}. Namely, 
generalizing Kuhn's 
observation \cite{Kuhn_annals}, examples can be manufactured by considering the 
topological 
realization of the beginning of an injective resolution (modulo nilpotents) 
of a given reduced module and forming the homotopy cofibre:
\[
\mathrm{hocofib} \Big\{ \bigvee_i \Sigma^n BV(i)_+ \rightarrow 
\bigvee_j\Sigma^n 
 BV(j)_+ 
\Big\}
\]
where $\{ V(i)\}$, $\{V(j)\}$ are finite sets of finite rank elementary abelian 
$p$-groups. The 
Artinian conjecture ensures that all finitely cogenerated modules admit such 
finite type presentations
 (modulo nilpotents).

Kuhn's non-realization conjectures highlight the interest of the case where 
$\rho_n M$ is finitely generated over the Steenrod algebra. If $M$ is 
$n$-connected,  
 the cases which have already been proved show that $\nlp_{n+1}M$ must actually 
be large; for example, $M$ cannot itself be finitely generated under these 
hypotheses. These results rely upon 
 Lannes' $T$-functor technology to reduce to the smallest non-trivial case 
\cite{CGPS}.
 
 The purpose of this paper is to show how the analysis of the first two 
non-trivial columns of the spectral sequence associated to the Arone-Goodwillie 
tower for the functor $ X \mapsto \Sigma^\infty \Omega^n X$ 
 impose conditions on the first two non-trivial layers of the nilpotent
filtration of 
$H^*(X)$. In the case $n=1$, this is an application of the Eilenberg-Moore 
spectral sequence 
and the result recovered is a generalization 
 of the main result of \cite[Section 6]{CGPS}. The case $n >1$ is new and arose 
as an 
offshoot of the author's programme to express the non-realization results 
obtained by Kuhn \cite{Kuhn_nonrealization} (for $p=2$) and  
 by Büscher, Hebestreit, Röndigs and Stelzer \cite{BHRS} (for odd primes), in 
terms of
obstruction 
classes in suitable $\ext$ groups. The key result is provided, on passage to 
the quotient category $\unst/ \nil$, by a 
 generalization of a theorem of Kuhn (see Theorem \ref{thm:split_mono_gen}) 
showing that certain $\ext$ groups in $\unst/ \nil$ are non-trivial.
 
An application of the case $n=1$ is the following:

\begin{THM}
 \label{THM:USigmaM}
 Let $M$ be a connected unstable module over $\field_2$ of finite type such 
that 
$\rho_0 M$ is non-zero and finitely generated. 
 Then $U (\Sigma M)$, the Massey-Peterson enveloping algebra of the suspension 
of $M$, is not realizable as the $\field_2$-cohomology of a space.
 \end{THM}

 The enveloping algebra $U (\Sigma M)$ is isomorphic to the exterior algebra 
$\Lambda^* (\Sigma M)$, so this   is a case of the following Theorem 
 (for a more refined statement, see Corollary \ref{cor:refined_COR}), 
in which $QK$ denotes the 
module of indecomposables:

\begin{THM}
\label{THM:K}
Let $K$ be a connected unstable algebra of finite type over $\field_2$ such 
that 
the cup square $Sq_0$ acts trivially on the augmentation ideal $\overline{K}$. 
If $QK$ is a $1$-connected unstable module such that 
$\rho_1 QK$ is non-zero and finitely generated,  then $K$ is not realizable as 
the  
$\field_2$-cohomology of a space. 
\end{THM}

Note that the  non-realization result of Gaudens and Schwartz  \cite{GS12,CGPS} 
does 
not apply directly here, since no restriction is placed upon the higher 
nilpotent filtration of $QK$.

\begin{rem}
The hypothesis that $\rho_1 QK$ is a finitely generated unstable module is 
essential. For example, consider $H^* (X)$ for  $X=SU$, so that $\Omega X 
\simeq 
 BU$. 
\end{rem}

 The main results of the paper are  Theorem \ref{thm:main} (for $n>1$) and 
Theorem \ref{thm:main_n=1} (for $n=1$). For these, the prime is taken to be 
two; 
the modifications necessary in the odd primary case are indicated in Section 
\ref{sect:podd}.

To give an idea of the flavour of the results, consider the
following:
 
 \begin{THM}
  For $1 \leq n \in \nat$ and $X$ an $n$-connected space such that $H^* (X)$ is 
of finite type and $\rho_n H^* (X)$ is finitely generated and non-trivial, 
  $$H^* (X) / \nlp_{n+2} H^* (X) $$
  is not the $n$-fold suspension of an unstable module.
 \end{THM}

At first sight,  this result may not appear surprising:  the structure theory 
of unstable 
algebras  (modulo nilpotents) implies that, under the given hypotheses, $\rho_n 
H^* (X)$ 
cannot be finitely generated if $X$ is the $n$-fold 
suspension of a connected space. The theorem shows that this is already 
exhibited algebraically by the  structure of $H^* (X) / \nlp_{n+2} H^* (X) $. 

This is a fundamental point: in the spectral sequence the columns will not in 
general be  unstable modules. For $n=1$, this is not a 
serious difficulty, since it is known that the columns of the $E_2$-term  of 
the 
Eilenberg-Moore spectral sequence are 
unstable. For the general case, this is no longer true, yet the spectral  
sequence converges to an unstable 
module. To allow the nilpotent filtration of unstable modules to be brought to 
bear, the notion of an {\em almost} unstable module 
is introduced here, which is shown to be sufficient to  cover the cases of 
interest.

Theorem \ref{thm:main} and Theorem \ref{thm:main_n=1} are much more precise, 
relating 
$\rho_n H^* (X)$ and  the next layer of the nilpotent filtration, $\rho_{n+1} 
H^* (X)$. Roughly speaking, for $n >1$ the result states that 
$\rho_{n+1}H^* (X)$ is at least as large as $\rho_n H^* (X)$; this is exhibited
by the injectivity  (modulo smaller objects relative to the Krull 
filtration of $\unst$) of the natural
transformation
\[
 \Phi \rho_n H^* (X) \rightarrow \rho_{n+1} H^* (X)
\]
that arises from the non-exactness of the iterated loop functor 
$\Omega^n : \unst \rightarrow \unst$.

For $n=1$, the  result is stronger, here the relevant transformation is
\[
 S^2 (\rho_1 H^* (X)) \rightarrow \rho_2 H^* (X),
\]
induced by the cup product of $H^* (X)$. The approach is unified here, showing 
how
the two cases are related. 
 
 This gives information on the beginning of the 
nilpotent filtration, whereas the proofs of the known cases of Kuhn's 
non-realization conjectures 
give global information. Where applicable, Lannes' mapping space technology 
can be used to study the higher parts of the nilpotent filtration; 
the ideas
involved  will be transparent to the experts and to the readers of \cite{CGPS} 
and are not developed here.

\bigskip
{\bf Organization of the paper:}
 Background is surveyed in Section \ref{sect:alg_prelim}; readers should
consult this as and when is necessary. 
 The technical notion of an almost unstable module is introduced in Section
\ref{sect:almost}; this is necessary to be able to control the 
 image of differentials in the spectral sequence, as is explained in Section
\ref{sect:auss}. The spectral sequence derived from 
 the Arone-Goodwillie tower is reviewed in Section \ref{sect:tower} and the
calculational input is provided in Section \ref{sect:cohomEP},
 namely the calculation of the $E_1$-term of the spectral sequence via the
cohomology of extended powers. In particular, it is shown that, for the case 
 at hand, the columns of the $E_1$-term are almost unstable. The case of the
second extended power admits an explicit algebraic model, as explained in
Section \ref{sect:alg_model}, which not only makes the results of the previous 
section more explicit in this case, but also provides a model for the 
differential $d_1$. The input from homological algebra is explained in 
Section \ref{sect:essential},
 refining a theorem of Kuhn; this is the key ingredient in the proofs of the 
 two main theorems, which are given in Section \ref{sect:main}.  Section 
\ref{sect:podd}
  sketches the modifications required in the odd primary case.

\section{Algebraic preliminaries}
\label{sect:alg_prelim}

This section reviews background, referring to the literature (in
particular \cite{schwartz_book} and \cite{K14}) for details. 
As usual, $\unst$ denotes the full subcategory of unstable modules in  $\amod$,
the category of graded modules over the mod $p$ Steenrod algebra 
$\cala$. The inclusion $\unst \rightarrow \amod$ has left adjoint $\destab 
: \amod \rightarrow \unst$, the destabilization functor. 

The suspension functor $\Sigma : \amod \rightarrow \amod$ restricts to $\Sigma
: 
\unst \rightarrow \unst$ and the iterated suspension functor $\Sigma^t : \unst 
\rightarrow \unst$ ($t \in \nat$) 
 has left adjoint the iterated loop functor $\Omega^t : \unst \rightarrow
\unst$, 
which identifies with the composite functor $\destab \Sigma^{-t}$ restricted to 
$\unst$.

The category
of unstable algebras is denoted $\unstalg$
and the Massey-Peterson enveloping algebra  $U : \unst \rightarrow \unstalg$ 
is 
the
left adjoint to the forgetful 
functor $\unstalg \rightarrow \unst$;  $U$ takes values in the
category $\unstalgaug$ of augmented unstable algebras.  The indecomposables 
functor $Q: \unstalgaug \rightarrow \unst$ is given 
explicitly by $Q K
:= \overline{K} / \overline{K}^2$, 
where $\overline{K}$ is the augmentation ideal.

\subsection{The nilpotent filtration}

The category of unstable modules $\unst$ has nilpotent filtration:
\[
 \ldots \subset \nil_{i+1} \subset \nil_i \subset \ldots \subset \nil_1 \subset 
\nil_0 = \unst,
\]
where $\nil_s$ is the smallest localizing subcategory containing all 
$s$-fold suspensions \cite{schwartz_book,K14}. In particular
$\nil_1$ is the  subcategory of nilpotent unstable modules  $\nil$.

The inclusion $\nil_s \hookrightarrow \unst$ admits a right adjoint $\nlp_s : 
\unst \rightarrow \nil_s \subset \unst$ so that an unstable module $M$ 
has a natural, convergent decreasing filtration:
\[
 \ldots \subset \nlp_{s+1} M \subset \nlp_s M \subset \ldots \subset \nlp_0 M = 
M
\]
and, for $s \in \nat$,  $\nlp_s M /\nlp_{s+1} M \cong \Sigma^s  \rho_sM$, where 
$\rho_s M$ is a 
reduced unstable module\footnote{The notation $\rho_s$ is used to avoid 
possible confusion with the Singer 
functors.}. (An unstable module is reduced if it contains no 
non-trivial suspension.)

\subsection{Functors between $\field$-vector spaces}
\label{subsect:F}

Functors on vector spaces over a finite field $\field$ arise naturally in the 
study of unstable 
modules via   Lannes' $T$-functor \cite{schwartz_book}.

\begin{nota}
 For $\field$ a finite field, let $\f$ denote the category of functors from
finite-dimensional $\field$-vector spaces to $\field$-vector
spaces and $\f_\omega \subset \f$ the full subcategory of locally finite (or
analytic) functors. 
\end{nota}

The category $\f$ is tensor abelian with enough projectives and injectives. 
A functor is finite if it has a finite composition series and is locally finite 
if it is the colimit of its finite subobjects. 

A  functor $F$ is polynomial of degree $d$ if $\Delta
^{d+1} F=0$, where $\Delta:\f
\rightarrow \f$ is the difference functor defined by 
$\Delta F(V):= F(V \oplus \field)/ F(V)$.

Over the prime field $\field_p$, the quotient category $\unst / \nil$ is
equivalent to $\f_\omega$  and the localization
functor $\unst \rightarrow \unst /\nil$ gives the exact functor $l : \unst
\rightarrow \f$ which can be identified in terms of Lannes' 
$T$-functor as $M \mapsto \{V \mapsto (T_V M )^0 \}$ \cite{schwartz_book}.

\begin{nota}
For $d \in \nat$, denote by $\f_d$ the full subcategory of $\f$ of functors of
polynomial degree $d$.
\end{nota}

\subsection{Examples of polynomial functors}
\label{subsect:basic_functors}

A number of polynomial functors arise here, which are closely
related to the $n$th tensor power functor 
$T^n : V \mapsto V^{\otimes n}$. The symmetric group $\symm_n$ acts naturally by
place permutations on $T^n$, 
giving:
\begin{enumerate}
 \item 
 the $n$th divided power $\Gamma^n := (T^n)^{\symm_n}$;
 \item 
 the $n$th symmetric power $S^n:= (T^n)/\symm_n$.
\end{enumerate}
These functors are dual under the Kuhn duality functor $D : \f\op \rightarrow
\f$, given by $DF (V):= F(V^*) ^*$ (see \cite{KI}).

Similarly, there is the $n$th exterior power functors $\Lambda ^n$, which  is
self-dual. The functors $S^1, \Lambda^1 , \Gamma^1$ coincide  with $\id : V 
\mapsto V$. 

The Frobenius $p$th power map induces a natural transformation $S^n \rightarrow
S^{np}$. Henceforth taking  $p=2$,  there is a non-split short exact sequence 
\begin{eqnarray}
\label{eqn:frob_ses}
 0 \rightarrow \id \rightarrow S^2 \rightarrow \Lambda^2 \rightarrow 0,
\end{eqnarray}
representing a non-zero class  $\phi \in \ext^1_\f (\Lambda^2 , \id)$. 

The composite of $S^2 \twoheadrightarrow \Lambda^2$ with its dual is the norm
map $S^2 \rightarrow \Gamma^2$; this occurs in the top row of the following
pullback diagram of exact sequences:
\begin{eqnarray}
 \label{eqn:ext2_pullback_diag}
 \xymatrix{
 0
 \ar[r]
 &
 \id 
 \ar[r]
 \ar@{=}[d]
 &
 S^2 
 \ar[r]
 \ar@{=}[d]
 &
 \Gamma^2 
 \ar@{^(->}[d]
 \ar[r]
 &
 \id 
 \ar@{^(->}[d]
 \ar[r]
 &
 0\\
 0
 \ar[r]
 &
 \id 
 \ar[r]
 &
 S^2 
 \ar[r]
&
T^2 
\ar[r]
&
S^2 
\ar[r]
&
0.
 }
\end{eqnarray}

The rows represent non-zero classes in $\ext^2 _\f$:
\[
 \tilde{e}_1 \in \ext^2_\f (S^2, \id) \mapsto e_1 \in \ext^2 _\f(\id, \id).
\]
These classes appear for example in \cite{FLS}.

\subsection{The Krull filtration}

The category $\unst$ of unstable modules has Krull filtration:
\[
 \unst_0 \subset \unst_1 \subset \ldots \subset \unst_n \subset \ldots \subset
\unst
\]
(see \cite{schwartz_book,K14});  $\unst_0$ identifies as  the full subcategory 
of
locally finite modules. For current purposes, the following is the key result:

\begin{prop}
\label{prop:krull_poly}
\cite{schwartz_book}
For $d \in \nat$, the functor $l : \unst \rightarrow \f$ restricts to 
 \[
  l : \unst_d \rightarrow \f_d.
 \]
Moreover, if $M$ is a reduced unstable module, then $M \in \obj \unst_d$ if and
only if $f(M)$ has polynomial degree $d$.
\end{prop}

\subsection{Functors on $\cala$-modules}

The categories $\amod$ and $\unst$ are tensor abelian; in particular, for $n \in
\nat$, the $n$th tensor functor 
$T^n : \amod \rightarrow \amod$, $M \mapsto M^{\otimes n}$ is defined, which
restricts to $T^n : \unst \rightarrow \unst$. 
Again,   $\symm_n$ acts naturally by place permutations on $T^n$, 
giving the $n$th symmetric invariants 
$\Gamma^n := (T^n)^{\symm_n}$  and the $n$th symmetric
coinvariants $S^n:= (T^n)/\symm_n$. The functors $\Gamma^n, S^n : \amod 
\rightarrow \amod$ restrict to
$\Gamma^n, S^n : \unst \rightarrow \unst$.
Similarly, the exterior power functor $\Lambda ^n : \amod \rightarrow
\amod$ restricts to $\Lambda^n : \unst \rightarrow \unst$. 

For the remainder of the section, the prime $p$ is taken to be $2$. 
Thus, the Frobenius functor $\Phi : \amod \rightarrow \amod$  \cite[Section
1.7]{schwartz_book} is the usual degree-doubling functor.

\begin{lem}
 \label{lem:Phi}
 For $p=2$, the Frobenius functor $\Phi : \amod \rightarrow \amod$  is exact,
commutes with tensor products and 
 there is a natural isomorphism 
 \[
  \Phi \Sigma \cong \Sigma^2 \Phi.
 \]
Moreover, $\Phi$ restricts to $\Phi : \unst \rightarrow \unst$ 
and, if $M$ is a reduced unstable module, $\Phi M$ is reduced.

For $i \in \nat$, $\Phi :\unst \rightarrow \unst$ restricts to:
$
 \Phi : \nil_i \rightarrow \nil_{2i}.
$
\end{lem}

The Frobenius short exact sequence (\ref{eqn:frob_ses}) and its dual have
analogues in $\amod$:

\begin{lem}
\label{lem:U_ses_Frobenius}
For $p=2$ and $M \in \obj \amod$, there are natural short exact sequences 
\begin{eqnarray*}
&& 0 \rightarrow \Lambda^2 M \rightarrow \Gamma^2 M \rightarrow \Phi M
\rightarrow 0 
\\
&& 0 \rightarrow \Phi M  \rightarrow S^2 M \rightarrow \Lambda^2 M\rightarrow 0 
.
\end{eqnarray*}
\end{lem}

\subsection{The Singer functors}
(In this section, to simplify presentation, $p$ is taken to be $2$.) 
The Singer functor $R_1$  was introduced for unstable modules in the form used 
here by Lannes and Zarati \cite{LZ}. For $M$ an unstable module, $R_1 M$ is the 
sub $\field [u]$-module of $\field [u] \otimes M$ generated by the image of the 
total Steenrod square 
$\mathrm{St}_1 (x) := \sum_i u^{|x|-i} \otimes Sq^i x \in \field [u] \otimes M$,
 for $x \in M$. A key fact is that $R_1 M$ is stable under the action of 
$\cala$ on $\field [u] \otimes M$.

The extension to all $\cala$-modules requires $\field [u] \otimes M$  to be 
replaced by a half-completed tensor product, 
since the sum in $\mathrm{St}_1 (x)$ is no longer finite in general. This is 
reviewed in \cite{p_viasm} and details 
are given (for odd primes) in \cite{p_destab}.

\begin{rem}
This Singer functor  must not be confused with the stabilized
version which occurs in \cite[Chapter II, Section 
5]{BMMS}, following 
ideas of Miller. For the current situation, compare the treatment (in 
homology) in \cite{KMcC}.
\end{rem}

\begin{nota}
Write $\field [u]$ for the unstable algebra generated by $u$ of degree $1$ and
$\field [u]\dash \amod $ for the category of $\field[u]$-modules in $\amod$
(respectively $\field[u]\dash \unst$ for 
$\field[u]$-modules in $\unst$). 
\end{nota}

\begin{lem}
The categories $\field[u]\dash \amod$ and $\field [u]\dash \unst$ are abelian
and there are exact forgetful functors
$  \field[u]\dash \amod \rightarrow  \field[u]\dash \unst$, \ \ 
 $\field[u]\dash \amod \rightarrow \amod$ and 
 $\field[u]\dash \unst \rightarrow \unst.$ 
\end{lem}

\begin{prop}
\cite{LZ,p_destab,p_viasm}
\begin{enumerate}
 \item 
The functor 
$ R_1 : \amod \rightarrow \field [u] \dash \amod
 $ is exact and  restricts to an exact functor $R_1 : \unst \rightarrow 
\field[u]\dash 
\unst$. 
\item
There is a natural surjection $R_1 \twoheadrightarrow \Phi$ which, 
for $M \in \obj \amod$, fits into
a
short exact sequence:
 \[
  0 \rightarrow u R_1 M \rightarrow R_1 M \rightarrow \Phi M \rightarrow 0. 
 \]
In particular, the $u$-adic filtration of $R_1M $  has
filtration quotients of the form $\Sigma^j \Phi M$. 
\end{enumerate}
\end{prop}

\begin{rem}
Forgetting the $\field[u]$-module structure, the Singer functor can be 
considered as an exact functor 
$R_1 : \amod \rightarrow \amod$ which  restricts to $R_1 : \unst \rightarrow 
\unst$.
\end{rem}

\begin{nota}
\cite{p_viasm}
 For $n \in \nat$ and $M \in \obj \amod$, write:
 $$R_{1/n} M := (R_1 
M)\otimes_{\field[u]} \field[u]/(u^n).$$
 \end{nota}

\begin{lem}
\label{lem:R_truncated}
 For $1 \leq n \in \nat$, $R_{1/n}$ is an exact functor $R_{1/n} : \amod 
\rightarrow \amod$ which restricts to $R_{1/n} : \unst 
\rightarrow \unst$.

For $M \in \obj \amod$
\begin{enumerate}
 \item 
 for $n \geq 2$, there is a natural short exact
sequence 
 \[
  0 
  \rightarrow 
  \Sigma^{n-1} \Phi M \rightarrow R_{1/n}M \rightarrow R_{1/n-1} M   \rightarrow
0;
 \]
 \item 
 the $u$-adic filtration of $R_{1/n}M$ has filtration
quotients
 $
  \Sigma^i \Phi M$ ,  $0 \leq i < n$;
  \item 
  the natural surjection $R_1 M \twoheadrightarrow
\Phi
M$ factors across $R_1 M \twoheadrightarrow R_{1/n}M$ as 
 $$
 R_{1/n}M\twoheadrightarrow \Phi M,
$$
which is an isomorphism mod $\nil$.
\end{enumerate}
\end{lem}

\section{Almost unstable modules}
\label{sect:almost}

Suppose that $\field=\field_2$.  (The results of
this section have analogues for odd primes.)

\begin{defn}
\label{defn:alunst}
 An unstable module $M \in \obj\unst$ is almost unstable if it admits a finite 
filtration with subquotients of the form 
 $\Sigma^{-i} N$ for some $i \in \nat$ and $N \in \nil_i$.

The full subcategory of almost unstable modules in $\amod$ is denoted 
$\alunst$. 
\end{defn}

\begin{prop}
\label{prop:alunst}
 There are inclusions of subcategories $\unst \subset \alunst \subset \amod$
and 
$\alunst$ is 
 an abelian Serre subcategory. 
 
 Moreover, 
 \begin{enumerate}
  \item 
   $\alunst$ is closed under $\otimes$ and hence under $\Sigma$;
   \item 
   any almost unstable module is concentrated in non-negative degrees; 
   \item 
   \label{item:alunst_bounded}
   a module $M \in \obj \amod$ concentrated in non-negative degrees
and  bounded above is almost unstable.
   \end{enumerate}
\end{prop}

\begin{proof}
It is clear that $\alunst$ contains $\unst$. To show that $\alunst$ is a Serre 
subcategory, it suffices to show that 
it is closed under formation of subobjects and quotients, since closure under 
extension is clear. 

For $M \in \obj \alunst$, write $f_i M $ for an increasing finite filtration 
that satisfies the defining property of Definition \ref{defn:alunst}. 
If $K \subset M$ is a submodule, consider the induced filtration
$f_i 
K := K \cap f_i M$. Then, by construction, 
$f_i K/f_{i-1}K \hookrightarrow f_i M/f_{i-1} M$. By hypothesis the right hand 
module is of the form $\Sigma^{-t} N$ for some $t \in \nat$ and $N \in \nil_t$;
 since $\nil_t$ is a Serre subcategory \cite{schwartz_book}, $f_iK $ is a 
filtration of the required form. 
 
 Similarly, for $M \twoheadrightarrow Q$, consider the 
quotient filtration $f_iQ := \mathrm{image} \{ f_i M \rightarrow Q \}$. Then
 $f_i Q /f_{i-1} Q$ is a quotient of $f_i M  /f_{i-1} M$, whence the result as 
before, {\em mutatis mutandis}.

 Closure under tensor product is a consequence of the fact that $\otimes$ 
restricts to $\otimes : \nil_i \times \nil_j \rightarrow \nil_{i+j}$ 
\cite{schwartz_book,K14}. 
Closure under $\Sigma$ follows since the suspension functor  identifies with 
$\Sigma \field \otimes - $.

The remaining statements are straightforward.
\end{proof}

\begin{exam}
\label{exam:almost}
As usual, extend the unstable algebra structure of $\field [u]$ to an algebra 
structure
in $\amod$ on $\field [u^{\pm 1}]$. Consider the submodule $M:= \field[u^{\pm 
1}]_{\geq -1}  \subset \field 
[u^{\pm 
1}] \in \obj \amod$, so that 
 $\Sigma M$ occurs in the short exact sequence:
 \[
  0 
  \rightarrow 
  \Sigma \field [u]
  \rightarrow 
  \Sigma M 
  \rightarrow 
  \field
  \rightarrow 
  0.
 \]
This exhibits $\Sigma M$ as an almost unstable module, whereas $M$ is not,
since 
it is non-zero in degree $-1$. Note that  $\Sigma^{t} M$ ($t \in \zed$) is 
never 
unstable. 
\end{exam}

\begin{rem}
 Proposition \ref{prop:alunst} implies that $\alunst$ is closed under the 
formation of finite limits and finite colimits. 
 \begin{enumerate}
  \item 
  That closure under inverse limits fails in general is clear from  Proposition 
\ref{prop:alunst} (\ref{item:alunst_bounded}), since {\em any} $\cala$-module 
$M$
   is the inverse limit of its system of truncations $M^{\leq k}$ (the quotient 
of $M$ by elements of degrees $>k$) as $k \rightarrow \infty$.
   \item 
   Closure under colimits also fails in general, as exhibited by the following 
example. For $0< t \in \nat$,  let $N(t)$ denote the subquotient $\Phi^t F(1) / 
\Phi^{2t} F(1)$ of the free 
unstable module $F(1)$; this has total dimension $t$, with 
classes in degrees $2^i$, $t \leq i < 2t$, linked by the operation 
$Sq_0$. 
   
   Consider the $\cala$-module:
   \[
    M:= \bigoplus _{t>1} \Sigma^{t -2^t} N(t).
   \]
Proposition \ref{prop:alunst} (\ref{item:alunst_bounded}) implies that each $ 
\Sigma^{t -2^t} N(t)$ is almost unstable  (and the choice of the desuspension 
ensures that $M$ is of finite type). 

However, $M$ is not almost unstable; if it were, there would exist $T \in \nat$ 
such that $\Sigma ^T M$ admits a finite filtration (say of length $l \in \nat$) 
such that each subquotient is unstable. 
Choosing $t \in \nat$ such that $t> l$ and  $t -2^t + T  <0$, consideration of  
the factor $ \Sigma^{t -2^t} N(t)$ leads to a contradiction.
\end{enumerate}
\end{rem}

\subsection{The nilpotent filtration of $\alunst$}

The above notions can be refined by introducing an analogue of the nilpotent 
filtration of $\unst$. 

\begin{defn}
 For $i \in \nat$, let $\alnil_i \subset \alunst$ be the full subcategory of 
objects which  admit a finite filtration with subquotients of the form 
 $\Sigma^{-t} N$ for some $t \in \nat$ and $N \in \nil_{i+t}$.
\end{defn}

By definition, there is a decreasing filtration:
\[
 \ldots \subset \alnil_{t+1} \subset \alnil_t \subset \ldots \subset \alnil_0 = 
\alunst. 
\]

Proposition \ref{prop:alunst} generalizes to:

\begin{prop}
\label{prop:alnil_serre_sigma}
 For $s, t \in \nat$:
\begin{enumerate}
 \item 
 $\alnil_t$ is a Serre subcategory of $\alunst$;
 \item 
 tensor product restricts to $\otimes : \alnil_s \times \alnil_t 
\rightarrow \alnil_{s+t}$:
\item 
suspension induces $\Sigma : \alnil_s \rightarrow \alnil_{s+1}$
which is an equivalence of categories, with inverse $\Sigma^{-1}$.
\end{enumerate}
 \end{prop}

\begin{proof}
Once established that $\Sigma$ induces an equivalence of categories, the
properties follow from the case of $\alunst = \alnil_0$.

To show that $\Sigma^{-1} : \amod \rightarrow \amod$ induces a functor
$\alnil_{s+1} \rightarrow \alnil_s$,  since $\Sigma^{-1}$ is exact, it suffices
to check on an almost unstable module 
of the form $\Sigma^{-i} N$ with $i \in \nat$ and $N \in \nil_{s+1 +i}$. 
 Then $\Sigma^{-1} (\Sigma^{-i} N)$ can be written $\Sigma^{-(i+1)} N$ with $N$ 
considered as lying in $\nil_{s+(i+1)}$.
\end{proof}

The category $\alunst$ is not stable under  $\Sigma^{-1}$. In combination with 
Proposition \ref{prop:destab_alnil} below, the above result should be compared 
with 
 the fact \cite{schwartz_book} that the loop functor $\Omega : \unst
\rightarrow 
\unst$ restricts to $\Omega : \nil_{i+1} \rightarrow \nil_i$.

\begin{prop}
\label{prop:alunst_stable_Phi_R}
For $i \in  \nat$,
 the Frobenius functor $\Phi : \amod \rightarrow \amod$ restricts to
 \[
  \Phi : \alnil_i \rightarrow \alnil_{2i}.
 \]
Hence, for $1 \leq n \in \nat$, the truncated Singer functor $R_{1/n}: \amod 
\rightarrow \amod$   restricts to:
 \[
  R_{1/n} : \alnil_i \rightarrow \alnil_{2i}.
 \]
\end{prop}

\begin{proof}
 The functors considered are exact, hence it suffices to consider behaviour on
a 
module of the form $\Sigma^{-t} N$ with $N \in \obj \nil_{i+t}$. 
 Now $\Phi \Sigma^{-t} N \cong \Sigma^{-2t} \Phi N$ and $\Phi N \in
\nil_{2(i+t)}$ by Lemma \ref{lem:Phi}, 
which implies the first statement. The corresponding statement for 
 $R_{1/n}$ then follows using the  $u$-adic
filtration (cf. Lemma \ref{lem:R_truncated}).
\end{proof}

\begin{prop}
\label{prop:destab_alnil}
 For $i \in \nat$, the destabilization functor restricts to 
 \[
  \destab : \alnil_i \rightarrow \nil_i.
 \]
\end{prop}

\begin{proof}
 The category $\nil_i$ is localizing and $\destab$
is 
right exact, hence it suffices to consider $\destab$ applied 
 to a module of the form $\Sigma^{-t} N$ with $N \in \obj \nil_{i+t}$. By 
construction, the composite functor $\destab \Sigma^{-t}$ restricted 
 to $\unst$ is the iterated loop functor $\Omega^t$. Since $\Omega^t$ restricts 
to 
 $\Omega ^t : \nil_{i+t} \rightarrow \nil_i$ \cite{schwartz_book}, this 
establishes the result.
\end{proof}

Recall that $\nlp_i : \unst \rightarrow \nil_i \subset \unst$ denotes the right 
adjoint to $\nil_i \hookrightarrow \unst$. 

\begin{cor}
\label{cor:nili_adjunction}
 For $M \in \obj \alnil_i $ and $N \in \obj \unst$, the inclusion  $\nlp_i N 
\hookrightarrow N$ induces an isomorphism
 \[
  \hom_{\amod} (M, \nlp_i N) \stackrel{\cong}{\rightarrow} \hom_{\amod} (M, N). 
 \]
\end{cor}

\begin{proof}
 Since $N$ and $\nlp_i N$ are unstable, the morphism identifies with 
  \[
  \hom_{\amod} (\destab M , \nlp_i N) {\rightarrow} \hom_{\amod} (\destab M, 
N). 
 \]
Proposition \ref{prop:destab_alnil} shows that $\destab M \in \obj \nil_i$, 
whence the result.
 \end{proof}

 As a particular case of Corollary \ref{cor:nili_adjunction}, one obtains: 
 
 \begin{cor}
 \label{cor:alnil_1_red}
 For $M \in \obj \alnil_1 $ and $N \in \obj \unst$ a reduced unstable module, 
 \[
  \hom_{\amod} (M,  N)=0.
 \] 
 \end{cor}

 \subsection{Good almost unstable modules}

For $M \in \obj \alunst$, there is a canonical surjection to a reduced unstable 
module, namely:
\[
 M \twoheadrightarrow (\destab M)/\nlp_1(\destab M). 
\]
In many cases of interest, the kernel of this map lies in $\alnil_1$.
This  motivates the following:

\begin{defn}
\label{defn:good_alunst}
 A module $M \in \obj \amod$ is a good almost unstable module if it is almost 
unstable and the kernel of  $M \twoheadrightarrow (\destab M)/\nlp_1(\destab M)$
lies in $\alnil_1$.
\end{defn}

\begin{lem}
 \label{lem:good_equivalent}
 A module $M \in \obj \amod$ is a good almost unstable module if and only if 
the 
kernel of $M \twoheadrightarrow \destab M$ lies in $\alnil_1$. 
\end{lem}

\begin{proof}
By definition $\nlp_1(\destab M)$ lies in $\nil_1$, whence the result.
\end{proof}

\begin{nota}
\label{nota:good_alunst}
  If $M \in \obj \alunst $ is good, 
 write the associated short exact sequence:
 \[
  0 
  \rightarrow M' 
  \rightarrow M 
  \rightarrow 
  \rho_0 M 
  \rightarrow 0,
 \]
where $M' \in \obj \alnil_1$ and $\rho_0 M \in \obj \unst$ is reduced.
\end{nota}

\begin{exam}
 Every unstable module is good when considered as an almost unstable module.
More generally, if $M$ is 
 of the form $\Sigma^{-t}N$ with $N \in \nil_t$, then $M$ is good almost
unstable, with associated exact sequence:
 \[
0
\rightarrow 
\Sigma^{-t} \nlp_{t+1} N
\rightarrow 
M 
\rightarrow 
\rho_t N
\rightarrow 
0.
 \]
In particular, if $t=0$ (so that $M$ is unstable), there is no conflict with the
notation $\rho_0 M $.  
\end{exam}

 \begin{prop}
 \label{prop:good_alunst}
  A subquotient of a good almost unstable module is good almost unstable.
  Moreover, if $f : M \twoheadrightarrow Q $ is a surjection from a good almost 
unstable module, then 
  $f$ induces a surjection $\rho_0 f : \rho_ 0 M \twoheadrightarrow \rho_0 Q$ 
which is an isomorphism if and only if 
  $\ker f$ lies in $\alnil_1$. 
  \end{prop}

  \begin{proof}
  Let $M$ be a good almost unstable module with associated short exact sequence 
as in Notation \ref{nota:good_alunst}. 
  Consider a submodule $K$; setting $K' := K \cap M'$, one has the 
morphism of short exact sequences:
  \[
   \xymatrix{
   0 
   \ar[r]
   &
   K' \ar[r]
   \ar@{^(->}[d]
   &
   K 
   \ar[r]
   \ar@{^(->}[d]
   &
K''
   \ar[r]
   \ar@{^(->}[d]
   &
   0
   \\
   0
   \ar[r]
   &
   M' 
   \ar[r]
   &
   M 
   \ar[r]
   &
   \rho_0M
   \ar[r]
   &
   0,
   }
  \]
which shows that $K'' \subset \rho_0 M $ is a reduced unstable module and $K' 
\subset M'$ belongs to $\alnil_1$, thus $K$ is good. 

Similarly, for a surjection $M \twoheadrightarrow Q$, there is a morphism of 
short exact sequences:
 \[
   \xymatrix{
   0 
   \ar[r]
   &
   M' \ar[r]
   \ar@{->>}[d]
   &
   M
   \ar[r]
   \ar@{->>}[d]
   &
\rho_0 M 
   \ar[r]
   \ar@{->>}[d]
   &
   0
   \\
   0
   \ar[r]
   &
   \tilde{Q}
   \ar[r]
   &
   Q
   \ar[r]
   &
   Q''
   \ar[r]
   &
   0,
   }
  \]
  where $\tilde{Q}$ is defined by the commutative square on the left, hence 
belongs to $\alnil_1$ since $M'$ does, and $Q''$ is unstable, as a quotient of 
$\rho_0 M$. 
Now $Q''$ is a quotient of $\destab Q$, by Lemma \ref{lem:good_equivalent}, 
thus 
 $Q$ is good.

By construction there is a surjection $\rho_0 M \twoheadrightarrow \rho_0 Q$. 
The final statement is clear.
\end{proof}

\begin{prop}
 \label{prop:good_alunst_stability}
 The class of good almost unstable modules is stable under finite direct sums
and under $\otimes$.
 Moreover, it is preserved by the functors:
 \begin{enumerate}
  \item 
  $\Phi : \alunst \rightarrow \alunst$;
  \item 
  $R_{1/n} : \alunst \rightarrow \alunst$, for $1 \leq n \in \nat$.
 \end{enumerate}
If $M$ is good almost unstable, then 
 \begin{eqnarray*}
  \rho_0 (\Phi M) & \cong & \Phi (\rho_0 M) \\
  \rho_0 (R_{1/n} M) & \cong & R_{1/n} (\rho_0 M) \ \cong  \ \Phi (\rho_0 M).
 \end{eqnarray*}
\end{prop}

\begin{proof}
 Straightforward. The statement for $\Phi$ and $R_{1/n}$ is a generalization of
Proposition \ref{prop:alunst_stable_Phi_R}, using 
 the fact that $\Phi$  preserves reduced unstable modules. 
\end{proof}

\section{Almost unstable spectral sequences}
\label{sect:auss}

The main interest in this section is in spectral sequences which
converge to an unstable module and the following 
natural question: to what extent can the theory of unstable modules be used to
understand the structure of the spectral sequence?

\begin{hyp}
\label{hyp:ss}
Suppose that the spectral sequence $(E_r^{*,*} , d_r)$ satisfies the following
conditions:
\begin{enumerate}
 \item 
it  is second quadrant ($E_r ^{s,t} = 0$ 
if $s >0$ or $t<0$) and cohomological $d_r : E_r ^{*,*} \rightarrow E_r^{*+r, 
*+1-r}$;
\item 
$E_1^{0,*}=0$; 
 \item 
 each $E_r^{-k, *}$ ($k \in \nat$) is an $\cala$-module and $d_r$ is 
$\cala$-linear, namely:
 \[
  d_r : \Sigma (\Sigma^{-k} E_r^{-k,*} ) 
  \rightarrow 
   (\Sigma^{-k+r} E_r^{-k+r,*} ) 
 \]
is a (degree zero) morphism of $\amod$;
 \item 
 the spectral sequence converges strongly and $\bigoplus _k \Sigma^{-k} 
E_\infty^{-k,*}$ 
is the associated graded of an unstable module, in particular each $\Sigma^{-k} 
E_\infty^{-k,*}$ is unstable.
\end{enumerate}
\end{hyp}

\begin{defn}
 A spectral sequence $(E_r^{*,*} , d_r)$ satisfying Hypothesis \ref{hyp:ss} is 
almost unstable (respectively good almost unstable) if 
 $\Sigma^{-k} E_1 ^{-k, *}$ is almost unstable (resp. good almost unstable) for 
all $k \in \nat$. 
\end{defn}

\begin{prop}
\label{prop:good_alunst_ss}
For a spectral sequence $(E_r^{*,*} , d_r)$ satisfying Hypothesis \ref{hyp:ss}, 
which is good almost unstable, 
and $k \in \nat$, 
\begin{enumerate}
 \item 
  $\Sigma^{-k} E_r ^{-k, *}$ is a good unstable module for all $1 \leq r \in 
\nat$;
  \item 
  for $1 \leq r \in \nat$,  $\rho_0 (\Sigma^{-k} E_{r+1}^{-k, *}) \subset
\rho_0 
(\Sigma^{-k} E_{r}^{-k, *})$ 
  with equality if $r>k$; in particular,  $\rho_0 (\Sigma^{-k} E_{\infty}^{-k, 
*}) = \rho_0 (\Sigma^{-k} E_{r}^{-k, *})$ for $r \geq k$. 
\end{enumerate}
\end{prop}

\begin{proof}
 The first statement follows from Proposition \ref{prop:good_alunst}. 
 
 For the second, the differential is of the form 
 \[
  d_r : \Sigma (\Sigma^{-k} E_r^{-k,*} ) 
  \rightarrow 
   (\Sigma^{-k+r} E_r^{-k+r,*} ), 
 \]
where $\Sigma (\Sigma^{-k} E_r^{-k,*} ) \in \obj \alnil_1$ and $ \Sigma^{-k+r} 
E_r^{-k+r,*} \in \alunst$ is a good unstable module. 
In particular, by Corollary \ref{cor:alnil_1_red}, the composite map 
 \[
  d_r : \Sigma (\Sigma^{-k} E_r^{-k,*} ) 
  \rightarrow 
   (\Sigma^{-k+r} E_r^{-k+r,*} )
   \twoheadrightarrow 
   \rho_0 (\Sigma^{-k+r} E_r^{-k+r,*} )
 \]
 is trivial, thus the image of $d_r$ lies in $(\Sigma^{-k+r} E_r^{-k+r,*} )'$. 
 
 It follows from the final statement of Proposition \ref{prop:good_alunst} that 
$\rho_0 (\Sigma^{-k} E_{r+1}^{-k,*} )$ identifies with $\rho_0 (\ker d_r)$, 
which is a submodule of $\rho_0 (\Sigma^{-k} E_{r}^{-k, *})$,
 by the argument employed in the proof of {\em loc. cit.}.

 The second point follows since the spectral sequence is concentrated in the 
second quadrant with trivial column $E^{0,*}_1$, by hypothesis. 
 \end{proof}

 \begin{cor}
 \label{cor:good_auss_d1}
  For a spectral sequence $(E_r^{*,*} , d_r)$ satisfying Hypothesis 
\ref{hyp:ss}, which is good almost unstable, 
  such that $E_1^{-1,*} = \Sigma (\Sigma^{-t} M) $ for some $t \in \nat$ and $M
\in \nil_t$,
\begin{eqnarray*}
 \rho_0 (\Sigma^{-1} E_\infty ^{-1,*})& = & \rho_0 (\Omega^t M)\ \cong \ \rho_t 
(M) 
\\
  \rho_0 (\Sigma^{-2}E_\infty ^{-2,*})& = & \rho_0 (\Sigma^{-2} E_2 ^{-2,*})
\end{eqnarray*}
Moreover, the differential $d_1 : \Sigma (\Sigma^{-2} E_1^{-2,*} ) 
  \rightarrow 
   \Sigma^{-1} E_1^{-1,*} = \Sigma^{-t} M$ factors across the inclusion 
   \[
    \Sigma^{-t} \nlp_{t+1} M \hookrightarrow \Sigma^{-t} M
   \]
and induces a morphism $d_1 : \rho_0 (\Sigma^{-2} E_1^{-2,*} ) \rightarrow 
\rho_{t+1} M$ and 
\[
 \rho_0 (\Sigma^{-2}E_\infty ^{-2,*}) \cong \ker \{ \rho_0 (\Sigma^{-2} 
E_1^{-2,*} ) \rightarrow \rho_{t+1} M\}. 
\]
 \end{cor}

 \begin{proof}
  The first part  follows  from Proposition 
\ref{prop:good_alunst_ss}. 
  By hypothesis, $(\Sigma^{-2} E_1 ^{-2,*})$ is a good almost unstable 
module, hence there is a short exact sequence
  \[
   0 
   \rightarrow 
   (\Sigma^{-2} E_1 ^{-2,*})'
   \rightarrow
   (\Sigma^{-2} E_1 ^{-2,*})
   \rightarrow
\rho_0  (\Sigma^{-2} E_1 ^{-2,*})
\rightarrow
0
\]
 with $(\Sigma^{-2} E_1 ^{-2,*})' \in \obj \alnil_1$. The differential 
$d_1$ is of the form:
 \[
  d_1 : \Sigma (\Sigma^{-2} E_1 ^{-2,*})
  \rightarrow 
  \Sigma^{-t} M 
 \]
hence $\Sigma^t d_1 : \Sigma^{t+1} (\Sigma^{-2} E_1 ^{-2,*})
  \rightarrow  M$. As $(\Sigma^{-2} E_1 ^{-2,*})$ is good almost 
unstable, 
   $\Sigma^{t+1} (\Sigma^{-2} E_1 ^{-2,*})$ lies in $\alnil_{t+1}$ and 
   $\Sigma^{t+1} \big((\Sigma^{-2} E_1 ^{-2,*})'\big)$ in $\alnil_{t+2}$. 
  Since $M$ is unstable, Proposition \ref{prop:destab_alnil} implies that 
$\Sigma^t d_1$ maps to $\nlp_{t+1} M$ and 
   its restriction to $\Sigma^{t+1} (\Sigma^{-2} E_1 ^{-2,*})'$ maps to 
$\nlp_{t+2}M$. 
   
The result follows as in the proof of Proposition \ref{prop:good_alunst_ss}.
 \end{proof}

\section{The Arone-Goodwillie spectral sequence}
\label{sect:tower}

In this section, the presentation of \cite{Kuhn_nonrealization}
is followed, since the results of  Section \ref{sect:cohomEP} use   {\em loc. 
cit.}.

For $X$ a pointed space (respectively a spectrum), the Arone-Goodwillie tower
associated to the functor $X \mapsto \Sigma^\infty \Omega^n X$ for  $n \in \nat$
 has the following form:
\[
 \xymatrix{
 & \ & \ar[d]
 \\
 &&P^n _3 X
 \ar[d]
 \\
 &&
 P^n _2 X 
 \ar[d]
\\
 X 
\ar[rruu]|{\epsilon_3}
 \ar[rru]|{\epsilon_2} 
 \ar[rr]|{\epsilon_1} 
 &&
 P^n_1 X, 
 }
\]
where $P^n _1 X= \Sigma^{-n} X$ for $n <\infty$ (ie the spectrum $\Sigma^{-n}
\Sigma^{\infty} X$) and  $P^n_0 = 0$.

Ahearn and Kuhn \cite{AK} identify the fibres of the tower  in terms of the
extended power construction via the cofibre sequence:
\[
 D_{n,j} \Sigma^{-n} X \rightarrow P^n_j X \rightarrow P^n_{j-1} X
\]
for $1 \leq j \in \nat$, where, for a spectrum $Y$,
\[
 D_{n,j} Y:= \Big(\lc(n,j)_+\wedge Y^{\wedge j} \Big) _{h \symm_j},
\]
$\lc(n,j)$ the Boardman-Vogt space of $j$ little $n$-cubes in an $n$-cube. 

For a pointed space $X$, the adjunction unit $X \rightarrow \Omega \Sigma X$
induces an $n$-fold loop map
$\Omega^n X \rightarrow \Omega^{n+1} \Sigma X$, for $n \in \nat$; by
\cite[Corollary 1.2]{AK}, this induces a natural map of towers
$P^n_\bullet X \rightarrow P^{n+1}_{\bullet}\Sigma X$ which identifies on the
level of the fibres as the natural transformation $D_{n,j} \Sigma^{-n} X
\rightarrow D_{n+1,j} \Sigma^{-n}X$ induced by the inclusion 
$\lc (n,j) \hookrightarrow \lc(n+1,j)$. Similarly, the natural evaluation map 
$\epsilon : \Sigma \Sigma^\infty 
\Omega^{n+1} X \rightarrow \Sigma^\infty \Omega^n X$ 
induces a map of towers $\Sigma P^{n+1}_\bullet X \rightarrow P^{n}_\bullet X$ 
and, on  fibres,
$
 \epsilon : \Sigma D_{n+1,j} X \rightarrow D_{n,j} \Sigma X
$
 (see \cite{AK}).

If $X$ is $n$-connected for $n \in \nat$, the connectivity of the maps 
$\epsilon_j$ increases linearly with
$j$, hence:

\begin{prop}
 \cite{Kuhn_nonrealization} 
For $X$ an $n$-connected space with $H^* (X)$ of finite type, the spectral 
sequence associated
to the Arone-Goodwillie tower satisfies Hypothesis \ref{hyp:ss} with 
$$E^{-j, *}_1 = H^* (\Sigma^j D_{n,j} \Sigma^{-n} X)$$
and converges strongly  to $H^* (\Omega^n X)$.

The associated filtration of $H^* (\Omega^n X) $ is
\[
 0 = F_0 H^*(\Omega^n X) \subset  F_1 H^*(\Omega^n X)\subset F_2 H^*(\Omega^n X)
\subset \ldots \subset H^*(\Omega^n X), 
\]
where $F_j H^* (\Omega^n X) = \mathrm{image} \{ H^* (P^n _j X) \rightarrow H^*
(\Omega^n X) \} $.
\end{prop}

There is a commutative diagram in $\unst$, in which  $ Q H^* (X_+)$ denotes the 
module of indecomposables of the unstable algebra $H^* (X_+)$:
\[
 \xymatrix{
 \Omega^n Q H^* (X_+) \ar[rr] 
 \ar@{->>}[rd]
 &&
 H^* (\Omega^n X) .
 \\
 &
 F_1 H^* (\Omega^n X) 
 \ar@{^(->}[ur] 
 }
\]

The  functor $U: \unst \rightarrow \unstalg$ induces a morphism of unstable 
algebras:
\begin{eqnarray}
\label{eqn:Umap} 
 U (\Omega^n Q H^* (X_+) ) \rightarrow H^* ((\Omega^n X)_+).
\end{eqnarray}

If $M$ is connected, there is a natural inclusion of unstable modules 
\[
 M \rightarrow \overline{UM} \subset UM 
\]
which induces a surjection onto the indecomposables $Q (UM)$; the product of
$UM$ induces an increasing filtration of the augmentation ideal $\overline{U
M}$:
\[
 M = F_1 \overline{U M} \subseteq F_2 \overline{UM} \subseteq F_3 \overline{UM}
\subseteq \ldots \subseteq F_j \overline{UM} \subseteq \ldots \subseteq
\overline{U M}.
\]

The results of Ahearn and Kuhn \cite{AK} imply that this filtration is 
compatible with the
filtration $F_j H^* (\Omega^n X)$; namely, for $1 \leq j \in \nat$, the 
morphism (\ref{eqn:Umap})
restricts to a morphism of unstable modules:
\[
 F_j \overline{U \Omega^n Q H^* (X_+) } \rightarrow F_j H^* (\Omega^n X).
\]

At the prime $p=2$, it is the submodule $F_2 \overline{UM}$ which is of
interest. The construction of $UM$ implies the following:

\begin{lem}
\label{lem:F_2UM}
For $M \in \obj \unst$ and $\lambda : \Phi M \rightarrow M$ the morphism of
unstable modules induced by $Sq_0$, 
$F_2 \overline{UM}$ occurs in the pushout of short exact sequences:
\[
 \xymatrix{
\Phi M 
\ar[d]_\lambda
\ar[r]
&
S^2 M 
\ar[r]
\ar[d]
&
\Lambda^2 M 
\ar@{=}[d]
\\
M 
\ar[r]
&
F_2 \overline{UM}
\ar[r]
&
\Lambda^2 M .
}
\]
 \end{lem}

\begin{rem}
\label{rem:cup_product}
Taking $M= F_1 H^* (\Omega^n X)$, one obtains the fundamental
morphism of short exact sequences:
\begin{eqnarray}
\label{eqn:filt_ss}
 \xymatrix{
\ \ \ &
F_1 H^* (\Omega^n X)
\ar[r]
\ar@{=}[d] 
&
F_2 \overline{U F_1 H^* (\Omega^n X)} 
\ar[r]
\ar[d]^{\cup}
&
\Lambda^2 ( F_1 H^* (\Omega^n X)) 
\ar[d]^{\overline{\cup}}
\\
&
F_1 H^* (\Omega^n X)
\ar[r]
&
F_2 H^* (\Omega^n X)
\ar[r]
&
F_2 H^* (\Omega^n X)/ F_1 H^* (\Omega^n X).
}
\end{eqnarray}

The identification of $\overline{\cup}$ in terms of the structure of the
spectral sequence will be important in Section \ref{sect:main}.
\end{rem}

\section{Cohomology of extended powers at $p=2$}
\label{sect:cohomEP}

Fix an integer $n \geq 1$ and a spectrum $Y$. In \cite[Section
3]{Kuhn_nonrealization}, Kuhn
 describes the mod $2$ cohomology of the extended powers $H^ *(D_{n,j}
Y)$;  we follow {\em loc.
cit.} in considering only spectra with $H^* (Y)$ bounded below and of finite type.  

The structure of $H^ *(D_{n,\bullet}
Y)$ is  determined (see \cite[Theorem 3.14]{Kuhn_nonrealization}) in
terms of the following morphisms:

\begin{enumerate}
 \item 
The  product \cite[Definition 3.3]{Kuhn_nonrealization} 
\[
 \star : H^* (D_{n,i} Y) \otimes H^*(D_{n,j} Y) \rightarrow H^* (D_{n,i+j} Y),  
\]
which  is a morphism of $\cala$-modules and induces a  commutative
(bi)graded algebra structure on  $H^ *(D_{n,\bullet} Y)$.
\item 
For  $1 \leq j \in \nat$, the dual Browder
operation \cite[Definition 3.5]{Kuhn_nonrealization}
\[
 L_{n-1} : H^{*} (Y)^{\otimes j} \rightarrow H^* (\Sigma^{(1-j)(n-1)}D_{n,j}  
Y),
\]
which is $\cala$-linear.
\item 
The  dual Dyer-Lashof operations (for $r \geq 0$)  
\cite[Definition 3.2]{Kuhn_nonrealization}
\[
 \ddl _r : H^d (D_{n,j} Y) \rightarrow H^ {2d+r} (D_{n,2j} Y).
\]
The operation $\ddl_r$ is trivial for $r\geq n$  \cite[Proposition 
3.8]{Kuhn_nonrealization}.
\end{enumerate}

\begin{rem}
The dual Dyer-Lashof operations $\ddl_r$ are not $\cala$-linear, but  satisfy
Nishida relations. 
Moreover, the operation $\ddl_0$ is not $\field$-linear; the default of
linearity  is given by the interaction with the $\star$-product:
\[
 \ddl_0 (x+y) = \ddl _0 (x) + \ddl_0 (y) + x \star y.
\]
Thus $\ddl_0$ behaves like a divided square operation.
\end{rem}

\begin{lem} 
For $1 \leq j \in \nat$, the $\star$-product induces a morphism of
$\cala$-modules:
  \[
  \star:  \Lambda^2 H^* (D_{n,i}Y) \rightarrow H^* (D_{n,2i} Y).  
  \] 
\end{lem}

\begin{proof}
 Follows from \cite[Proposition 3.11 (iii)]{Kuhn_nonrealization}.
\end{proof}

\begin{lem}
\label{lem:images_ddl}
 For $1 \leq j \in \nat$ and an integer $l>0$, the sub vector space of $H^*
(D_{n,2j} Y) $ generated by the images of
\[
 \ddl _r : H^* (D_{n,j} Y) \rightarrow H^ {2*+r} (D_{n,2j} Y)
\]
($r \geq l$) is a sub $\cala$-module $\Big(\sum_{r \geq l} 
\mathrm{image}(\ddl_r)\Big)$
of $H^* (D_{n,2j} Y) $.
\end{lem}

\begin{proof}
 Follows from the Nishida relation given in \cite[Proposition
3.1(ii)]{Kuhn_nonrealization}.
\end{proof}

\begin{rem}
 The higher dual Dyer-Lashof operations $\ddl_r$, $r \geq n$ act trivially by
\cite[Proposition 3.8]{Kuhn_nonrealization}, hence the sum $\sum_{r \geq l}
\mathrm{image}(\ddl_r)$ is finite (and zero for
 $l \geq n$).
\end{rem}

\begin{prop}
\label{prop:filter_ddl}
  For $1 \leq j \in \nat$, the dual Dyer-Lashof operations induce $\cala$-linear
maps:
\begin{eqnarray*}
\overline{\ddl_0} &:& \Phi H^* (D_{n,j} Y) \rightarrow H^* (D_{n,2j} Y)/ \Big ( 
\Lambda^2
H^* (D_{n,j} Y) + \sum_{r \geq 1} \mathrm{image}(\ddl_r) \Big)
\\
\overline{\ddl_l} &:& \Sigma^l \Phi H^* (D_{n,j} Y) \hookrightarrow H^* 
(D_{n,2j} Y)/ \Big
( \sum_{r \geq l+1} \mathrm{image}(\ddl_r) \Big),
\end{eqnarray*}
where $l>0$.
\end{prop}

\begin{proof}
 By \cite[Proposition 3.11(i)]{Kuhn_nonrealization}, the operation $\ddl_0$
becomes $\field$-linear after the passage to the quotient by 
the submodule $\Lambda^2 H^* (D_{n,j} Y)$. Moreover, the Nishida relation for
$\ddl_0$ \cite[Proposition 3.15(i)]{Kuhn_nonrealization} establishes
the $\cala$-linearity, after passing to the additional quotient by the image of
the higher dual Dyer-Lashof operations, which is a sub $\cala$-module by Lemma
\ref{lem:images_ddl}.

The argument for $\ddl_l$ ($l>0$) is similar, using the Nishida relation
\cite[Proposition 3.1(ii)]{Kuhn_nonrealization}.
\end{proof}

\begin{thm}
\label{thm:cohomEP_good}
Suppose that $Y$ is a spectrum such that $H^* (Y) \in \obj \alunst$ is almost
unstable and is of finite type. 
Then for $1 \leq j \in \nat$:
\begin{enumerate}
 \item 
$H^* (D_{n,j} Y) \in \obj \alunst$ is almost unstable;
\item 
if $H^* (Y)$ is a good almost unstable module, then $H^* (D_{n,j} Y)$ is good
and 
\begin{enumerate}
\item
$
 \rho_0 (H^* (D_{n,j}Y))\cong \Gamma^j \rho_0 (H^* (Y))
$, $n \geq 2$; 
\item 
$\rho_0 (H^* (D_{1,j}Y))\cong T^j \rho_0 (H^* (Y)).
$
\end{enumerate}
\end{enumerate}
\end{thm}

\begin{proof}
The case $n=1$ is straightforward (compare \cite[Remark
2.1]{Kuhn_nonrealization}), hence suppose that $n \geq 2$. 

The proof that the modules are almost unstable is based on \cite[Theorem
3.14]{Kuhn_nonrealization}, which states that, 
$H^* (D_{n,\bullet} X) $ is generated as a (bi)graded
commutative 
algebra (under the $\star$-product) by elements of the form 
\[
 \ddl _{r_1} \ldots \ddl_{r_t} L_{n-1} (x_1 \otimes \ldots \otimes  x_k) \in H^*
(D_{n,2^t k}Y),
\]
subject to the relations given in \cite[Section 3.3]{Kuhn_nonrealization}.

The category $\alunst$ is a Serre subcategory of $\amod$ and is stable under
$\otimes$, thus 
an increasing induction upon $j$ implies that it is sufficient to work modulo
$\star$-decomposables.
Hence one is reduced to considering $\star$-indecomposables of the above form
(note that Adem-type 
relations intervene in considering the words in the dual Dyer-Lashof operations,
by \cite[Proposition 3.13]{Kuhn_nonrealization}).
Moreover, since the dual Dyer-Lashof operations double the $j$ degree and
$\ddl_r$ is trivial for $r\geq n$, in a given $j$-degree, 
there are only finitely many words $\ddl _{r_1} \ldots \ddl_{r_t}$ which arise.

For $k$ tensor factors, the morphism $L_{n-1}$ is $\cala$-linear:
\[
L_{n-1} : \Sigma^{(k-1)(n-1)} H^* (Y)^{\otimes k} \rightarrow H^* (D_{n,k} Y). 
\]
The hypothesis implies that $\Sigma^{(k-1)(n-1)} H^* (Y)^{\otimes k}$ lies in
$\alnil_{(k-1)(n-1)}$ (in particular is almost unstable), 
hence so does the image of $L_{n-1}$, by Proposition
\ref{prop:alnil_serre_sigma}. For $k>1$, 
$(k-1) (n-1)>0$, whereas for $k=1$, $L_{n-1}$ is an isomorphism.  

A straightforward filtration argument based on Proposition \ref{prop:filter_ddl}
allows words in dual Dyer-Lashof operations to be treated.
Namely, up to higher terms, the image of an operation $\ddl_r$ is a quotient of
the functor $\Sigma^r \Phi$ and, by Proposition \ref{prop:alunst_stable_Phi_R},
the functor $\Sigma^r \Phi$ induces:
\[
 \Sigma^r \Phi : \alnil_i \rightarrow \alnil_{2i+r}.
\]
Since $\alnil_{2i+r}$ is a Serre subcategory, up to filtration, this exhibits
the image under $\ddl_r$ of an element of $\alnil_i$ as lying in 
$\alnil_{2i+r}$. 

Putting these facts together, one concludes that $H^* (D_{n,j} Y)$ is almost
unstable. Moreover, the argument shows that the 
only possible contributions not in $\alnil_1$ arise from
$\star$-products of terms from the image of iterates of $\ddl_0$. 
If $H^* (Y)$ is good almost unstable, then there is an associated short exact
sequence in $\amod$:
\[
 0 
\rightarrow 
(H^* Y)' 
\rightarrow 
H^* Y
\rightarrow 
\rho_0 (H^* Y)
\rightarrow 
0.
\]
Any terms arising from $(H^* Y)'$ also lie in $\alnil_1$. Hence, to prove the
result, it suffices to show that the projection 
$H^* Y \twoheadrightarrow \rho_0 (H^* Y)$ induces a surjection (recall $n \geq 
2$, by hypothesis) 
\[
 H^* (D_{n,j} Y) \twoheadrightarrow \Gamma^j (\rho_0 (H^* Y)), 
\]
since $\Gamma^j (\rho_0 (H^* Y))$ is a reduced unstable module. This is clear as
graded vector spaces; to check that the morphism 
is $\cala$-linear, use \cite[Proposition 3.15(i)]{Kuhn_nonrealization}, which
shows that the action of the Steenrod squares 
on $\ddl_0$ is correct modulo the higher dual Dyer-Lashof operations.
\end{proof}

\begin{cor}
\label{cor:d1_niln}
For $1 \leq n \in \nat$ and  $X$  an $n$-connected space such that $H^*(X) \in 
\nil_n$ and is of finite type,  the
spectral sequence calculating $H^* (\Omega^n X)$ associated 
to the Arone-Goodwillie tower is good almost unstable.

In particular, the morphism $d_1$ from the $(-2)$-column to the $(-1)$-column
induces
\begin{eqnarray*}
 \Gamma^2 (\rho_n H^* (X)) 
\rightarrow 
\rho_{n+1} H^* (X) 
& \ \ & n \geq 2\\ 
T^2 (\rho_1 H^* (X)) 
\rightarrow 
\rho_2 H^* (X)
&&
n=1.
\end{eqnarray*}
\end{cor}

\begin{proof}
 The final statement follows from Corollary \ref{cor:good_auss_d1}.
\end{proof}

\section{Algebraic Models}
\label{sect:alg_model}

In the case of the second extended power, it is possible to give explicit 
algebraic models 
for their cohomology. For current purposes, this is not strictly necessary; it
is included 
since it makes the results of Section \ref{sect:cohomEP}  much more explicit.

\subsection{The case of the second extended power}

The calculation of $H^* (D_{\infty,2} Y) $  (see \cite{KMcC}, where homology is 
used) is a stable version of the
calculation of the quadratic construction \cite{Milgram,GLZ,HLS2}. Here $H^* 
(D_{n,2} Y) $ is considered for finite $n$; this brings the dual
Browder operations into the picture.

\begin{lem}
\label{lem:nat_incl}
 For $Y$ a spectrum with $H^* (Y)$ bounded below and of finite type,
 \begin{enumerate}
  \item 
  the $\star$-product induces a monomorphism of $\cala$-modules:
  \[
  \star:  \Lambda^2 H^* (Y) \hookrightarrow H^* (D_{n,2} Y);  
  \]
\item 
the dual Browder operation induces a monomorphism of $\cala$-modules:
\[
L_{n-1} : \Sigma^{n-1} S^2 H^* (Y) \hookrightarrow H^* (D_{n,2} Y);
\]
\item 
the sum of these induces a monomorphism of $\cala$-modules:
\[
 \star \amalg L_{n-1} :  \Lambda^2 H^* (Y) \oplus  \Sigma^{n-1} S^2 H^* (Y)
\hookrightarrow H^* (D_{n,2} Y).
\]
 \end{enumerate}
\end{lem}

\begin{proof}
 The morphisms are provided respectively by \cite[Proposition
3.11(iii)]{Kuhn_nonrealization} and \cite[Proposition
3.12]{Kuhn_nonrealization}. The injectivity is a consequence of 
 \cite[Theorem 3.14]{Kuhn_nonrealization}.
\end{proof}

Recall from Section \ref{sect:alg_prelim} that, for $M \in \obj \amod$ and $1
\leq n \in \nat$, there are natural morphisms:
\begin{eqnarray}
\label{eqn:pb_po}
 \xymatrix{
&
&
\Gamma^2 M 
\ar@{->>}[d]
\\
\Sigma^{n-1} \Phi M 
\ar@{^(->}[r]
\ar@{^(->}[d]
&
R_{1/n} M 
\ar@{->>}[r]
&
\Phi M
\\
\Sigma^{n-1} S^2 M .
}
\end{eqnarray}
Here the middle row is not in general a sequence (for $n=1$ the morphisms are
isomorphisms) and not in general exact (for $n \geq 3$).

\begin{defn}
For $1 \leq n \in \nat$ and $M \in \obj \amod$, let $\cale_n M$ denote the
$\cala$-module given by forming the pushout and pullback of diagram
(\ref{eqn:pb_po}).
\end{defn}

\begin{exam}
\label{exam:n=1}
 For $n=1$ and $M \in \obj \amod$, $\cale_1 M$ is naturally isomorphic to
$M^{\otimes 2}$. 
\end{exam}

\begin{prop}
\label{prop:functor_E}
 For $1 \leq n \in \nat$, the above construction defines
a functor $
 \cale_n : \amod \rightarrow \amod $
which restricts to $\cale_n : \unst \rightarrow \unst $.

For $M \in \obj \amod$, there is a natural short exact sequence:
  \[
   0 
   \rightarrow 
   \Lambda^2 M \oplus \Sigma^{n-1} S^2 M \rightarrow \cale_n M \rightarrow
R_{1/n-1}M  \rightarrow 0.
  \]
\end{prop}

\begin{proof}
 Straightforward.
\end{proof}

The functor $\cale_n$  provides an algebraic model for
$H^* (D_{n,2}Y)$:

\begin{prop}
\label{prop:alg_model_2col}
 For $n \in \nat \cup \{ \infty \}$ and  $Y$ a spectrum with $H^* (Y)$ bounded
below and of finite type, there is a natural isomorphism 
\[
 \mathscr{E}_n H^* (Y) \cong H^* (D_{n, 2} Y) 
\]
which extends the inclusion $\star \amalg L_{n-1} :  \Lambda^2 H^* (Y) \oplus  
\Sigma^{n-1} S^2 H^* (Y)
\hookrightarrow H^* (D_{n,2} Y)$ of Lemma \ref{lem:nat_incl}.
 \end{prop}

\begin{proof}
The result is essentially a restatement of the results of \cite[Section 
3]{Kuhn_nonrealization}, using the algebraic functors introduced in Section 
\ref{sect:alg_prelim}.
\end{proof}

\begin{exam}
For $n=1$, one recovers from Example \ref{exam:n=1} and Proposition 
\ref{prop:alg_model_2col} the standard identification $H^* (D_{1,2} X) \cong
H^* 
(X) ^{\otimes 2}$. 
\end{exam}

\begin{rem}
 For $n \geq 1$, there are natural transformations 
$\cale_{n+1} \rightarrow  \cale_n$,  $
\cale_n \Sigma  \rightarrow  \Sigma \cale_{n+1}$
that provide algebraic models (via Proposition \ref{prop:alg_model_2col}) for
the  morphisms in cohomology induced respectively by
$D_{n,2} Y  \rightarrow  D_{n+1,2}Y$ and $
\Sigma D_{n+1,2} Y \rightarrow  D_{n,2}\Sigma Y$.
\end{rem}

\subsection{The algebraic differential}

There is an algebraic differential which is related to the differential used by 
Singer (see \cite{p_viasm} for references). 

Recall  that  $\field [u^{\pm 1}]$ has an $\cala$-module structure extending 
that of $\field[u]$;  it is a fundamental fact that the residue
map 
$
 \field [u^{\pm 1}] \rightarrow \Sigma^{-1} \field
$ 
is $\cala$-linear. This gives rise to a
natural transformation 
$
 d_M : R_1 M \rightarrow \Sigma^{-1} M 
$ 
in $\amod$, since $R_1 M $ embeds in the half-completed tensor product
$\field[u^{\pm 1}] \largetensor M$. If $M$ is unstable then  $d_M$ is trivial.

\begin{lem}
\label{lem:diff_R1t}
\cite{p_viasm}
 For $1 \leq n \in \nat$ and $N \in \obj \unst$, the differential
$d_{\Sigma^{-n}N}$ induces a natural 
transformation $d_{1/n} : R_{1/n} (\Sigma^{-n}N )\rightarrow  \Sigma^{-n-1}N$
which fits into a commutative diagram
\[
 \xymatrix{
R_1 (\Sigma^{-n} N) \ar[r]^{d_{\Sigma^{-n}N}}
\ar@{->>}[d]
&
\Sigma^{-n-1}N
\ar@{=}[d]
\\
R_{1/n} (\Sigma^{-n}N )
\ar[r]_{d_{1/n}}
&
\Sigma^{-n-1}N.
}
\]
The cokernel of 
$\Sigma d_{1/n} : \Sigma R_{1/n} (\Sigma^{-n}N ) \rightarrow  \Sigma^{-n}N
$
is $\Omega^n N$.
\end{lem}

\begin{proof}
 Straightforward.
\end{proof}

By the definition of $\cale_n M$ (for general $M \in
\obj \amod$), the quotient $\cale_n M/ \Lambda^2 M$ occurs as the pushout of the
diagram:
\[
 \xymatrix{
\Sigma^{n-1} \Phi M \ar@{^(->}[r]
\ar@{^(->}[d]
&
R_{1/n} M 
\\
\Sigma^{n-1} S^2 M.
}
\]

\begin{prop}
\label{prop:model_d1}
 For $1 \leq n \in \nat$ and $K$ a connected unstable algebra with augmentation
ideal $\overline{K}$, the natural 
transformation 
$$
d_{1/n} :  R_{1/n} (\Sigma^{-n}\overline{K} ) \rightarrow  
\Sigma^{-n-1}\overline{K}$$
 together 
with the product 
$ 
 S^2 (\overline{K}) \rightarrow \overline{K} 
$ 
induce a natural transformation in $\amod$:
\[
 d_1 :  \cale_n (\Sigma^{-n} \overline{K}) \rightarrow \Sigma^{-n-1}
\overline{K}.
\]
\end{prop}

\begin{proof}
The subobject $\Sigma^{n-1} S^2(\Sigma^{-n} \overline{K}) $ of
$\cale_n(\Sigma^{-n} \overline{K}) $ is naturally 
isomorphic to $\Sigma^{-n-1}S^2 ( \overline{K})$, hence the product 
induces a natural morphism of $\cala$-modules
$$
\Sigma^{n-1} S^2(\Sigma^{-n} 
\overline{K}) \rightarrow \Sigma^{-n-1} \overline{K}.
$$
 The verification that this  is compatible with $d_{1/n}$ given by Lemma
\ref{lem:diff_R1t} 
is straightforward.
\end{proof}

\subsection{The spectral sequence differential $d_1$}

Consider the first stages of the Arone-Goodwillie tower for $\Omega^n X$, with 
$X$ an
$n$-connected space. There is a cofibre sequence of spectra
\[
 D_{n,2} \Sigma^{-n}X \rightarrow P^n_2 X \rightarrow \Sigma^{-n} X
\]
and the differential $d_1$ from the $-2$-column to the $-1$-column of the
spectral sequence is the connecting morphism
\[
 d_1 : \Sigma H^* ( D_{n,2} \Sigma^{-n} X) \rightarrow  H^* (\Sigma^{-n}
X).
\]
This can be identified algebraically in terms of the isomorphism of Proposition
 \ref{prop:alg_model_2col}.

\begin{prop}
\label{prop:compat_d1}
 For $1 \leq n \in \nat$ and $X$ a connected space with $H^* (X)$ of finite 
type, the following diagram
commutes:
\[
 \xymatrix{
\Sigma \mathscr{E}_n (\Sigma^{-n} H^* (X) ) 
\ar[d]_\cong 
\ar[r]^(.55){\Sigma d_1} 
&
 \Sigma^{-n} H^* (X) 
\ar[d]^\cong
\\
\Sigma
H^* ( D_{n,2} \Sigma^{-n} X)
\ar[r]_(.55){d_1}
&
 H^* (\Sigma^{-n} X)
}
\]
in which the $d_1$ of the top row indicates the algebraic differential of
Proposition \ref{prop:model_d1}. 
\end{prop}

\begin{proof}
This result corresponds to \cite[Proposition 4.3]{Kuhn_nonrealization}.
\end{proof}

\subsection{Exploiting the nilpotent filtration}
\label{subsect:exploit_nil}

\begin{prop}
\label{prop:alg_d1_niln}
 For $1 \leq n \in \nat$ and an unstable module $N \in \obj \nil_n$, the
algebraic 
differential $d_{1/n}:R_{1/n} (\Sigma^{-n} N) \rightarrow \Sigma^{-n-1} N$
factors 
across $\Sigma ^{-n-1} \nlp_{n+1}N$ and the resulting map fits into a natural
commutative diagram:
\[
 \xymatrix{
R_{1/n}(\Sigma^{-n}N)
\ar[r]
\ar@{->>}[d]
&
\Sigma ^{-n-1} \nlp_{n+1}N
\ar@{->>}[d]
\\
\Phi \rho_n N
\ar[r]_{\delta_n}
&
\rho_{n+1}N,
}
\]
where the  vertical morphisms are induced by the natural projections $R_{1/n}
\twoheadrightarrow \Phi$ of Lemma \ref{lem:R_truncated} together with 
$\Sigma^{-n}N \twoheadrightarrow \rho_n N$.

The natural transformation $\delta_n : \Phi \rho_n N \rightarrow \rho_{n+1}N$ is
induced by the linear transformation $x \in \Sigma^{-n} N \mapsto
Sq^{|x|+1}(x)$. 
In particular, if $\delta_n$ is non-trivial, then $N$ is not an $n$-fold
suspension. 
\end{prop}

\begin{proof}
 Straightforward,  unravelling 
definitions to identify  the morphism $\delta_n$.
\end{proof}

\begin{rem}
\label{rem:identify_delta}
 Using the notation of Proposition \ref{prop:alg_d1_niln}, if $x = \Sigma^{-n} 
y$ for $y \in N$, then 
 $Sq^{|x|+1} (\Sigma^{-n} y) = \Sigma^{-n} Sq^{|y|+1 -n}y$. In particular, for 
$n=1$, the map $\Phi \rho_1 N \rightarrow \rho_2 N$ is 
 induced by the operation $Sq_0$ on $N$.
\end{rem}

\begin{cor}
\label{cor:d1_ngeq2}
 In the situation of Corollary \ref{cor:good_auss_d1} for $n \geq 2$, the 
induced morphism
factors as 
\[
 \Gamma^2 (\rho_n H^* (X) ) 
\twoheadrightarrow 
\Phi (\rho_n H^* (X)) 
\stackrel{\delta_n}{\rightarrow} 
\rho_{n+1} H^* (X).
\]
\end{cor}

There is an alternative viewpoint on the natural transformation
$\delta_n: \Phi \rho_n N \rightarrow \rho_{n+1}N$ for $N \in \nil_n$, 
based on the following result, in which $\Omega^n_1 : \unst \rightarrow \unst$
denotes the first left derived functor of the iterated loop functor 
$\Omega^n$. 

\begin{prop}
\label{prop:Omega_n_1}
 For $1 \leq n \in \nat$ and $s \geq n$, the functor $\Omega^n_1 : \unst
\rightarrow \unst$ restricts to 
\[
 \Omega^n_1 : \nil_s \rightarrow \nil_{2 (s-n)+1}.
\]
Moreover, for $N \in \obj \nil_n$ (so that $\Omega^n_1 N \in \nil_1$) 
\[
 \rho_1 (\Omega^n_1 N) \cong \Phi (\rho_n N).
\]
\end{prop}

\begin{proof}
 The proof is by induction on $n$; for $n=1$ this is 
\cite[Lemma 6.1.3]{schwartz_book} (which also states that $\Omega_1$ 
takes values in $\nil_1$).
 The inductive step uses the short exact sequence $\Omega \Omega^{n-1}_1
\rightarrow \Omega^n_1 \rightarrow \Omega_1 \Omega^{n-1}$
associated to the identification $\Omega^n = \Omega \circ \Omega^{n-1}$ (see 
\cite{p_viasm} and the references therein). Namely, if $M \in \nil_{s}$ 
with $s \geq n$, then $\Omega^{n-1}_1 M \in \nil_{2 (s-(n-1))+1} = \nil_{2 
(s-n) 
+3}$, hence
  $\Omega \Omega^{n-1}_1 M \in \nil_{2(s-n)+2}$. Similarly, $\Omega^{n-1}M \in 
\nil_{s-(n-1)} = \nil _{(s-n)+1}$, hence $\Omega_1 \Omega^{n-1} M \in \nil_{2 
((s-n)+1)-1}= \nil_{2 (s-n) +1}$. 
The result follows, since $\nil_{2 (s-n) +1}$ is closed under extensions.
\end{proof}

\begin{cor}
\label{cor:les_Omegan}
 For $1 \leq n \in \nat$ and $N \in \obj \nil_n$, there is a natural exact
sequence
\[
 \Sigma \Phi \rho_{n} M 
\stackrel{\Sigma \delta_n} {\rightarrow} 
\Sigma \rho_{n+1} M 
\rightarrow 
\Omega^n(M/\nlp_{n+2}M) 
\rightarrow 
\rho_n M
\rightarrow 
0.
\]
\end{cor}

\begin{proof}
 This follows by considering the long exact sequence for $\Omega^n_\bullet$
associated to the short exact sequence
\[
 0
\rightarrow 
\Sigma^{n+1} \rho_{n+1} M 
\rightarrow 
M/\nlp_{n+2}M 
\rightarrow 
\Sigma^n \rho_n M
\rightarrow 
0
\]
together with the factorization provided by Proposition \ref{prop:Omega_n_1}.
\end{proof}

\section{Essential extensions}
\label{sect:essential}

Let $\field$ be a finite field  and recall that $\f$ is the
category
of functors from finite-dimensional $\field$-vector spaces to $\field$-vector
spaces.
 
The aim of this section is to give a generalization of the following
result:

\begin{thm}
 \cite[Theorem 4.8]{KIII}
 For $F$ a non-constant finite functor, precomposition with $F$ induces a 
(naturally
split) monomorphism:
 \[
  \ext_\f^* (G, H) 
  \hookrightarrow 
  \ext_\f^* (G \circ F, H \circ F).
 \]
\end{thm}

This is refined  by using the observation that the result only
depends on the top polynomial degree behaviour of $F$. 

\begin{rem}
For  $f : F_1 \rightarrow F_2$  a morphism between functors taking
finite-dimensional values, for any morphism 
$\alpha : G \rightarrow H$, there is a commutative diagram:
\[
 \xymatrix{
 G \circ F_1 
 \ar[r]^{\alpha_{F_1}}
 \ar[d]_{G f}
 &
 H \circ F_1
 \ar[d]^{Hf}
 \\
  G \circ F_2 
 \ar[r]_{\alpha_{F_2}}
 &
 H \circ F_2
 }
\]
which corresponds to the two (equivalent) ways to define a natural
transformation: $\hom_{\f} (G, H) \rightarrow \hom_{\f} (G\circ F_1, H\circ
F_2)$. 
By naturality this extends to $\ext^* _\f$.
\end{rem}

Recall that a functor $F$ has polynomial degree exactly $d$ if it is polynomial
of degree $d$ but not of degree $d-1$ (ie $F \in \f_d \backslash \f_{d-1}$).

\begin{thm}
\label{thm:split_mono_gen}
 For $f: F_1\rightarrow F_2$ a morphism between finite functors of 
polynomial degree exactly $d>0$,  if $\mathrm{image}(f)$ has polynomial degree 
exactly
$d$, then  precomposition with $F_1$ together with $f$  induce a (naturally 
split)
monomorphism:
 \[
  \ext_\f^* (G, H) 
  \hookrightarrow 
  \ext_\f^* (G \circ F_1, H \circ F_2).
 \]
\end{thm}

\begin{proof}
 The proof follows that of  \cite[Theorem 4.8]{KIII}, which relies upon the fact
that the full subcategory of functors of polynomial degree at most $1$ (which
contains the constant and additive functors) is semisimple. 
 The hypotheses ensure that $\Delta^{d-1}F_1$ and $\Delta^{d-1}F_2$ are
non-constant functors of polynomial degree $\leq 1$ and that $\Delta^{d-1}f$ maps to a non-constant functor. 
 
 Applying \cite[Lemma 4.12]{KIII}, there exists a finite functor $E$ such that
the identity functor $\id$ is a direct summand of $\mathrm{image}(f)
\circ E$. (The proof of the lemma is based on the fact that $\Delta^{d-1} F$ is a natural
direct summand 
of the functor $V \mapsto F (V \oplus \field^{d-1})$, for $F \in \obj \f$.)

By semi-simplicity the splitting factors:
 \[
  \xymatrix{
  \id
  \ar@{^(->}[r]
  \ar@/_2pc/[rrrr]^{1}
  & 
  F_1 \circ E 
  \ar@{->>}[r]
  \ar@/^1pc/[rr]^{f_E}
  &
  \mathrm{image}(f) \circ E
  \ar@{^(->}[r]
  &
  F_2 \circ E
  \ar@{->>}[r]
  &\id.
  }
 \]
The proof is completed, {\em mutatis mutandis}, as in \cite[Section 4.3]{KIII}.
\end{proof}

\begin{exam}
\label{exam:split_mono_ses_frob}
 Consider a finite functor $F$ of polynomial degree exactly $d>0$ and let
$q_{d-1} F$ be the largest quotient of $F$ of polynomial degree $\leq d-1$, with
 associated short exact sequence:
 \[
  0 
  \rightarrow 
  \overline{F} 
  \rightarrow 
  F
  \rightarrow 
  q_{d-1} F
  \rightarrow 
  0
 \]
where $\overline{F} \hookrightarrow F$ satisfies the hypotheses of Theorem
\ref{thm:split_mono_gen}. The theorem shows that:
 \[
  \ext_\f^* (G, H) 
  \hookrightarrow 
  \ext_\f^* (G \circ \overline{F}, H \circ F).
 \]
 is a split monomorphism.
 
 Taking $\field = \field_2$,  there is a  non-zero class $\phi \in \ext^1_\f
(\Lambda^2 , \id) $ given by the short exact sequence (\ref{eqn:frob_ses}).
Applying the above
result gives a pull-back diagram of short exact sequences
 \[
  \xymatrix{
  0\ar[r]
  &
  F 
  \ar@{=}[d]
  \ar[r]
  &
  \cale 
  \ar[r]
    \ar@{^(->}[d]
  &
  \Lambda^2 \circ \overline{F} 
  \ar[r]
  \ar@{^(->}[d]
  &
  0
  \\
  0\ar[r]
  &
  F 
  \ar[r]
  &
  S^2 \circ F 
  \ar[r]
  &
  \Lambda^2 \circ F 
  \ar[r]
  &
  0
  }
 \]
in which both short exact sequences are essential, in particular the top row
corresponds to a non-zero class of $  \ext_\f^1 (\Lambda^2 \circ \overline{F}, 
F)$.
\end{exam}

For $\field=\field_2$, recall the extension classes from Section
\ref{subsect:basic_functors}, which are related by the Frobenius morphism $\id \rightarrow
S^2$:
\[
\tilde{e}_1 \in \ext^2_\f (S^2, \id) \mapsto e_1 \in \ext^2 _\f(\id, \id).
\]

\begin{cor}
\label{cor:classes_S2_non-trivial}
For $\field= \field_2$ and  $\tilde{F} \subset F$  finite functors of polynomial degree exactly 
$d>0$, the following classes are non-trivial 
 \begin{enumerate}
  \item 
  $i^* (e_1 \circ F) \in \ext^2 _\f (\tilde{F} , F);$
  \item 
  $i^* (\tilde{e}_1 \circ F) \in \ext^2 _\f (S^2 \circ \tilde{F} , F),$
 \end{enumerate}
where $i^*$ denotes the pullback induced by the inclusion $i : \tilde{F}
\hookrightarrow F$.
 
Moreover, under the Frobenius morphism $\id \rightarrow S^2$, these classes are
related by 
\[
 i^* (\tilde{e}_1 \circ F) \in \ext^2 _\f (S^2 \circ \tilde{F} , F) \mapsto i^*
(e_1 \circ F) \in \ext^2 _\f (\tilde{F} , F).
\]
\end{cor}

\begin{proof}
 Follows directly from Theorem \ref{thm:split_mono_gen}, as in Example
\ref{exam:split_mono_ses_frob}.
\end{proof}

\section{The Main results}
\label{sect:main}

This section gives the proofs of the main results of the paper, Theorems 
\ref{thm:main} and \ref{thm:main_n=1}, 
based upon the non-triviality result of Section \ref{sect:essential}. Namely an 
obstruction class
$\omega_X $ living in a suitable group $\ext^2_\f (-,-)$ is introduced, which 
must vanish. Combined 
with the non-vanishing result Theorem \ref{thm:split_mono_gen}, this provides 
restrictions on the structure of 
$H^*(X)$, in particular on the relationship between the first two non-trivial 
layers of its nilpotent filtration.

\subsection{Compatibility with the $\star$-product}

In the following, note that \cite[Proposition 4.1]{Kuhn_nonrealization} shows 
that the spectral sequence associated to the Arone-Goodwillie tower is a 
spectral sequence 
of bigraded algebras with respect to the $\star$-product.

\begin{prop}
\label{prop:compatibility_star_cup}
 For $X$ an $n$-connected space with $H^* (X)$ of finite type, there are 
identifications via the 
Arone-Goodwillie spectral sequence calculating 
 $H^* (\Omega^n X)$:
 \begin{eqnarray*}
  F_1 H^* (\Omega^n X) &\cong& \Sigma E^{-1,*}_\infty \\
  F_2 H^* (\Omega^n X) / F_1 H^* (\Omega^n X) &\cong& \Sigma E^{-2,*}_\infty 
 \end{eqnarray*}
and the morphism 
\[
 \overline{\cup} : \Lambda^2 (F_1 H^* (\Omega^n X)) \rightarrow  F_2 H^* 
(\Omega^n X) / F_1 H^*( \Omega^n X)
\]
coincides with the morphism induced by the $\star$-product in the spectral 
sequence.
\end{prop}

\begin{proof}
 The result follows from the compatibility of the spectral sequence with the 
multiplicative structure, established by Ahearn and Kuhn \cite{AK},
  as used in \cite{Kuhn_nonrealization}.
 \end{proof}

 \subsection{Working modulo nilpotents}
 
 Throughout this section, the following is supposed: 

\begin{hyp}
\label{hyp:X_niln}
For fixed $1 \leq n \in \nat$, $X$ is an $n$-connected space  such that  $H^* 
(X)\in \nil_n$ and is of finite type.
\end{hyp}

The first condition ensures strong convergence of the spectral sequence 
calculating 
$H^* (\Omega^n X)$ and the 
second specifies the class of spaces of interest here. In particular, there is
a 
surjection 
$
 H^* (X) 
 \twoheadrightarrow 
 \Sigma^n \rho_n H^* (X).
$

\begin{nota}
 For $j \in \{1, 2 \}$, set 
$
  F_j := l\big( F_j H^* (\Omega^n X)\big).
  $
\end{nota}

Corollary \ref{cor:good_auss_d1}  implies:

\begin{lem}
\label{lem:identify_F1}
 There is a natural isomorphism 
 $
  F_1 \cong l \big(\rho_n H^* (X)\big).
 $
\end{lem}

For clarity of presentation, the case $n=1$ is postponed to Section 
\ref{subsect:n=1}.

\begin{nota}
 Set $K := l \big(\mathrm{ker} \{ \Phi \rho_n H^* 
(X)\stackrel{\delta_n}{\rightarrow} \rho_{n+1} H^* (X) \}\big)$, 
 so that $K \subset F_1$ 
 and write $\widetilde{\Gamma^2 \circ F_1}$ for the functor defined by the 
pullback 
diagram:
 \[
  \xymatrix{
  0
  \ar[r]
  &
  \Lambda^2\circ F_1 
  \ar[r]
  \ar@{=}[d]
  &
  \widetilde{\Gamma^2\circ F_1}
    \ar@{^(->}[d]
    \ar[r]
    &
K
\ar@{^(->}[d] 
  \ar[r]
  &
  0
  \\
  0
   \ar[r]
  &
  \Lambda^2\circ F_1 
  \ar[r]
  &
 \Gamma^2\circ F_1
    \ar[r]
    &
 F_1    
  \ar[r]
  &
  0
  }
 \]
\end{nota}

\begin{lem}
\label{lem:identify:F2/F1}
 For $n \geq 2$ there is an isomorphism
 $
  F_2/ F_1 \cong \widetilde{\Gamma^2 \circ F_1 } 
 $
and the morphism induced by $\overline{\cup} : \Lambda^2 ( F_1 H^* (\Omega^n
X)) 
\rightarrow  F_2 H^* (\Omega^n X) / F_1 H^* (\Omega^n X)$ identifies with 
the inclusion 
$
 \Lambda^2 \circ F_1 \hookrightarrow \widetilde{\Gamma^2 \circ F_1}.
$
\end{lem}

\begin{proof}
 The result follows from Corollary \ref{cor:good_auss_d1}, using the 
identification of the differential $d_1$ modulo nilpotents, which follows from 
Corollary \ref{cor:d1_ngeq2}.
\end{proof}

\begin{nota}
 For $n \geq 2$, let $\omega_ X \in \ext^2 _\f (K, F_1)$ be the
Yoneda 
product of the class $\phi\circ F_1 \in \ext^1 _\f (\Lambda^2 \circ F_1 , F_1)
$ 
with the class in $\ext^1_\f (K , \Lambda^2 \circ F_1 )$ representing 
$F_2/F_1$.
\end{nota}

\begin{lem}
\label{lem:omega_X_descriptions}
For $n \geq 2$, 
\begin{enumerate}
 \item 
$\omega_X = i^* (e_1 \circ F_1)$, where $i^* : \ext^2_\f (F_1, F_1) \rightarrow 
\ext^2 _\f (K, F_1)$ is induced by the inclusion $i : K 
\hookrightarrow F_1$; 
\item 
there is a commutative diagram in which the three-term rows and columns are 
short exact:
\[
 \xymatrix{
 F_1 
 \ar@{=}[d]
 \ar[r]
 &
 S^2\circ F_1
 \ar[r]
 \ar[d]
 &
 \Lambda^2 \circ F_1 
 \ar[d]
 \\
 F_1 
 \ar[r]
 &
 F_2 
 \ar[r]
 &
F_2/F_1
 \ar[d]
 \\
 &&
 K
 }
\]
and $S^2 \circ F_1 \rightarrow F_2$ is induced by $\cup : S^2 (F_1 H^*
(\Omega^n 
X) ) \rightarrow F_2 H^* (\Omega^n X)$.
\item 
In particular $\omega_ X = 0  \in \ext^2 _\f (K, F_1)$.
\end{enumerate}
\end{lem}

\begin{proof}
The first point follows from the identification  of $F_2/F_1$ given in Lemma 
\ref{lem:identify:F2/F1} and the definition of $\widetilde{\Gamma^2 \circ 
F_1}$. 

The second point is a consequence of the compatibility between the cup product 
and the $\star$-product and follows by combining Proposition 
\ref{prop:compatibility_star_cup} with 
the identifications of $F_1$ (Lemma \ref{lem:identify_F1}) and $F_2/F_1$. 

The final point follows from homological algebra.
\end{proof}

\begin{thm}
\label{thm:main}
 Suppose that $n \geq 2$ and $X$ is a topological space satisfying Hypothesis 
\ref{hyp:X_niln}. 
 If $F_1 = l \big(\rho_n H^* (X)\big)$ is a finite functor of polynomial degree 
exactly $d >0$, then 
$K:= l \Big(\mathrm{ker} \{ F_1 \stackrel{\delta_n}{\rightarrow} \rho_{n+1} H^* 
(X) \}\Big) $ has  polynomial degree $<d$.
 
 Equivalently, if $\rho_n H^* (X)$ is finitely generated over $\cala$ and lies 
in 
$\unst_d \backslash \unst_{d-1}$ for $d>0$, then 
 \[
  \delta_n : \Phi \rho_n H^* (X) 
  \rightarrow 
  \rho_{n+1} H^* (X)
 \]
has kernel in $\unst_{d-1}$, in particular $\delta_n$ is non-trivial.

Hence 
\begin{enumerate}
\item 
$ \rho_{n+1} H^* (X)\not \in \unst_{d-1}$;
 \item 
 $H^*(X) / \nlp_{n+2} H^* (X) $ is not an $n$-fold suspension. 
\end{enumerate}
\end{thm}

\begin{proof}
 Lemma \ref{lem:omega_X_descriptions} shows that the obstruction class 
$\omega_X$ is
trivial 
and also that it identifies with $i^* (e_1 \circ F_1)$. 
 If $F_1$ is a finite functor then, by  Theorem \ref{thm:split_mono_gen}, 
this class is non-trivial if $K$ has exact polynomial degree $d$. 

 The fact that $H^*(X) / \rho_{n+2} H^* (X) $ is not an $n$-fold suspension 
follows from the identification of $\delta_n$ in Corollary \ref{cor:les_Omegan}.
 \end{proof}

 \begin{cor}
 \label{cor:main_n>1}
  Let $M$ be an unstable module of finite type such that 
  \begin{enumerate}
   \item 
     $M \in \obj \nil_n$, for $2 \leq n \in \nat$;
  \item
$M$ is $n$-connected;
     \item 
     $\rho_n M$ is finitely generated over $\cala$ and lies in $\unst_d 
\backslash \unst_{d-1}$ for some $d >0$;
     \item 
     the morphism induced by $Sq_{n-1}$, $\Phi M \rightarrow \Sigma^{1-n}\big( 
M/ \nlp_{n+2} M \big)$ has image in $\unst_{d-1}$.
  \end{enumerate}
Then $M$ cannot be the reduced $\field_2$-cohomology of an $n$-connected space.
\end{cor}

 \begin{proof}
  A consequence of Theorem \ref{thm:main}, using the identification 
of $\delta_n$ from Section \ref{subsect:exploit_nil} (cf. 
  Remark \ref{rem:identify_delta}) to express the condition in terms of 
$Sq_{n-1}$.
 \end{proof}

 \begin{rem}
  The statement has an unavoidably technical nature, due to the identification 
of the morphism $\delta_n$ in terms of an operation $Sq_i$ with $i >0$. 
  In the case $n=1$, the analogous result is conceptually simpler, since the 
corresponding operation is the cup square $Sq_0$ (see Corollary
\ref{cor:refined_COR}).
 \end{rem}

 \subsection{The case $n=1$}
 \label{subsect:n=1}
 
 The exceptional case ($n=1$) corresponds to the 
Eilenberg-Moore spectral sequence. Theorem \ref{thm:cohomEP_good} implies that 
the cohomology 
 of $D_{1,2}(\Sigma^{-1} X)$ is the $\cala$-module $H^*(\Sigma^{-1} X)
^{\otimes 
2}$. The differential $d_1$ 
 is induced by the product $\mu : H^*( X)^{\otimes 2} {\rightarrow}
H^*(X)$, which factors by commutativity as 
 $
 S^2 (H^*( X)){\rightarrow}
H^*(X). 
  $ 
Since $H^* (X)$ is nilpotent, by hypothesis, this induces 
\[
S^2 ( \rho_1 H^* (X)) \rightarrow \rho_2 H^* (X). 
\]

\begin{nota}
 Write $K$ for the kernel of the corresponding morphism in $\f$:
 \[
 S^2 \circ  F_1 \rightarrow l \big(\rho_2 H^* (X)\big);
\]
\end{nota}

\begin{lem}
\label{lem:n=1_F2/F1}
 The quotient $F_2/F_1$ fits into the pullback diagram of short exact sequences:
  \[
  \xymatrix{
  0
  \ar[r]
  &
  \Lambda^2\circ F_1 
  \ar[r]
  \ar@{=}[d]
  &
F_2/F_1 
    \ar@{^(->}[d]
    \ar[r]
    &
  K    
   \ar@{^(->}[d] ^j 
  \ar[r]
  &
  0
  \\
  0
   \ar[r]
  &
  \Lambda^2\circ F_1 
  \ar[r]
  &
 T^2\circ F_1
    \ar[r]
    &
 S^2 \circ F_1
  \ar[r]
  &
  0
  }
 \]
\end{lem}

\begin{proof}
 Straightforward.
\end{proof}

 \begin{nota}
  For $n = 1 $, let $\omega_ X \in \ext^2 _\f (K, F_1)$ be the Yoneda product
of 
the class $\phi\circ F_1 \in \ext^1 _\f (\Lambda^2 \circ F_1 , F_1) $ with the 
class in $\ext^1_\f (K , \Lambda^2 \circ F_1 )$ representing $F_2/F_1$.
 \end{nota}

 The following result is the analogue for $n=1$ of Lemma 
\ref{lem:omega_X_descriptions}:
 
 \begin{lem}
\label{lem:omega_X_descriptions_n=1}
For $n =1$, 
\begin{enumerate}
 \item 
$\omega_X = j^* (\tilde{e}_1 \circ F_1)$, where $j^* : \ext^2_\f (S^2\circ F_1, 
F_1) \rightarrow \ext^2 _\f (K, F_1)$ is induced by the inclusion $j : K 
\hookrightarrow S^2 \circ F_1$; 
\item 
there is a commutative diagram in which the three-term rows and columns are 
short exact:
\[
 \xymatrix{
 F_1 
 \ar@{=}[d]
 \ar[r]
 &
 S^2 \circ F_1 
 \ar[r]
 \ar[d]
 &
 \Lambda^2 \circ F_1 
 \ar[d]
 \\
 F_1 
 \ar[r]
 &
 F_2 
 \ar[r]
 &
 F_2/F_1 
 \ar[d]
 \\
 &&K.
 }
\]
\item 
In particular $\omega_ X = 0  \in \ext^2 _\f (K, F_1)$.
\end{enumerate}
\end{lem}

\begin{proof}
Analogous to the proof of Lemma \ref{lem:omega_X_descriptions}, {\em mutatis 
mutandis}, using 
Lemma \ref{lem:n=1_F2/F1}. 
\end{proof}

Before stating the theorem, it is worth resuming the situation; 
$F_1$ is of polynomial degree exactly $d>0$ and $\tilde{F_1}
\stackrel{i}{\hookrightarrow} F_1$ 
an arbitrary subfunctor. There is a commutative square of inclusions:
\[
 \xymatrix{
 \tilde{F_1} 
  \ar@{.>}[rd] _\alpha
 \ar@{^(->}[rr]^i 
  \ar@{^(->}[dd]
  &&
  F_1 
  \ar@{^(->}[dd]
 \\
 &
 K 
 \ar@{^(->}[rd]_j
 \\
 S^2 \circ \tilde{F_1} 
 \ar@{-->}[ur]_\beta
 \ar@{^(->}[rr]_{S^2 i} 
 &&
  S^2 \circ F_1. 
 }
\]
Applying $\ext^2_\f (-, F_1)$ gives a diagram of $\ext$-groups. Of particular 
interest are the following 
observations:

\begin{enumerate}
 \item 
 If the dotted factorization $\alpha$ exists, then the class $\alpha^ *
(\omega_X) \in \ext^2 _\f (\tilde{F_1}, F_1)$ coincides with 
$i^* (e_1 \circ F_1)$, in the notation of Corollary
\ref{cor:classes_S2_non-trivial}.
\item 
If the dashed factorization $\beta$ exists (and hence $\alpha$), then the class
$\beta^ * (\omega_X) \in \ext^2 _\f (S^2 \circ \tilde{F_1}, F_1)$ coincides
with 
$i^* (\tilde{e}_1 \circ F_1)$, in the notation of Corollary
\ref{cor:classes_S2_non-trivial}.
\end{enumerate}

\begin{thm}
\label{thm:main_n=1}
Let  $X$ be a simply-connected space such that $H^* 
(X)$ is nilpotent and of finite type and  $F_1= l \big(\rho_1 H^* (X)\big)$ is 
finite  of polynomial degree
exactly $d 
>0$. For a subfunctor $\tilde{F_1}$ of $F_1$ of polynomial degree exactly $d$,
the following properties hold:
\begin{enumerate}
 \item 
 the subfunctor $S^2 \circ \tilde{F_1} \subset S^2 \circ F_1$ is not contained
within $K$, equivalently the morphism 
\[
 S^2 \circ \tilde{F_1} \rightarrow l \big(\rho_2 H^* (X))
\]
is non trivial;
\item 
the subfunctor $\tilde{F_1} \subset S^2 \circ F_1$ is not contained within $K$,
equivalently the composite:
\[
 \tilde{F_1} \hookrightarrow  S^2 \circ \tilde{F_1} \rightarrow l \big(\rho_2
H^* (X))
\]
is non-trivial.
\end{enumerate}
Hence 
\begin{enumerate}
 \item 
 the image of $S^2 \circ \tilde{F_1} \rightarrow l \big(\rho_2 H^* (X))$ has
polynomial degree exactly $2d$,  in particular $ \rho_2 H^* (X) \not\in
\unst_{2d-1}$;
 \item 
 the morphism induced by $Sq_0$:
 \[
  \Phi \rho_1 H^* (X) \rightarrow  \rho_2 H^* (X)
 \]
has kernel in $\unst_{d-1}$, in particular is non-trivial, so that $H^* (X)/
\nlp_3 H^* (X)$ is not a suspension.
\end{enumerate}
\end{thm}

\begin{proof}
 The argument follows the proof of Theorem \ref{thm:main}, {\em mutatis 
mutandis}.  

Suppose that $\tilde{F_1} \subset F_1$ is a subfunctor such that $S^2 \circ 
\tilde{F_1} \subset K$. 
Then the class $\omega_X \in \ext^2 _\f (K, F_1)$ pulls back to the class of 
$\ext^2_\f (S^2 \circ \tilde{F_1}, F_1)$ which is the  pullback  of 
$\tilde{e}_1 
\circ 
F_1$ 
under the morphism induced by $\tilde{F_1} \subset F_1$. Suppose that 
$\tilde{F}_1$ has polynomial degree exactly $d$, 
then this class is non-trivial, by Corollary \ref{cor:classes_S2_non-trivial}; 
this contradicts Lemma \ref{lem:omega_X_descriptions_n=1}.

The argument using the hypothesis that $
\tilde{F_1} \subset K$ is similar {\em mutatis mutandis}, again using Corollary
\ref{cor:classes_S2_non-trivial}.

To show that the polynomial degree of the image of $S^2 \circ \tilde{F_1} 
\rightarrow l \big(\rho_2 H^* (X))$ is exactly $2d$, without loss of generality 
we may assume that $\tilde{F_1}$ has no non-trivial quotient
of polynomial degree $<d$. In this case, the functor $S^2 \circ \tilde{F_1}$
has 
no non-trivial quotient of degree $<2d$ (see \cite{CGPS}, for example).

Finally, the composite $\tilde{F_1} \hookrightarrow  S^2 \circ \tilde{F_1}
\rightarrow l \big(\rho_2 H^* (X))$ 
is given,  upon passage to $\unst/ \nil$, by the morphism $\Phi \Sigma \rho_1 
H^* (X)
\rightarrow \Sigma^2 \rho_2 H^* (X)$
 induced by $Sq_0$, noting that the natural isomorphism $\Phi \Sigma \cong 
\Sigma^2 \Phi$ allows
 the suspensions to be removed. Taking $\tilde{F_1}:= K \cap F_1$, the above 
argument shows
that this has polynomial degree $\leq d-1$, whence the result.
\end{proof}

\begin{rem}
 Theorem \ref{thm:main_n=1} is an improvement upon the main result of 
\cite[Section 6]{CGPS}, where only the case $\tilde{F_1} = \overline{F_1}$, 
the kernel of $F_1 \twoheadrightarrow q_{d-1}F_1$ (see Example
\ref{lem:omega_X_descriptions_n=1} for this notation) is considered, for the 
morphism
 $S^2 \circ \overline{F_1} \rightarrow l (\rho_2 H^* (X))$. The above theorem 
 unifies this with the case $n \geq 2$. 

The improvement is obtained by using  Theorem \ref{thm:split_mono_gen} rather
than \cite[Theorem 4.8]{KIII}.
\end{rem}

\begin{cor}
\label{cor:refined_COR}
Let $K$ be a connected unstable algebra of finite type over $\field_2$ such 
that 
$\overline{K}$ is $1$-connected and $\rho_1 \overline{K}$ is finitely generated 
over $\cala$ and
lies in $\unst_d \backslash \unst_{d-1}$ for $d>0$ (in particular is non-zero). 

If the morphism induced by the cup square, $
Sq_0 : \Phi \overline{K} \rightarrow K / \nlp_3 K 
$, 
has image in $\unst_{d-1}$, then  $\overline{K}$ cannot be the reduced 
$\field_2$-cohomology of a simply-connected space. 
\end{cor}

\begin{proof}
Follows from Theorem \ref{thm:main_n=1} (cf. Corollary \ref{cor:main_n>1}). 
\end{proof}

\begin{rem}
Corollary \ref{cor:refined_COR} applies if the image of $Sq_0 : \Phi 
\overline{K} 
\rightarrow K$ lies in $\unst_{d-1}$. This occurs, for example, if $Sq_0$ is 
zero, the case considered 
in Theorem \ref{THM:K}. Under this hypothesis, $QK$ is a quotient of 
$\Sigma \Omega \overline{K}$, and it follows that $\rho_1 \overline{K}$ 
is isomorphic to $\rho_1 QK$.
\end{rem}

\section{The case $p$ odd}
\label{sect:podd}

This section  sketches the modifications necessary for $p$ an 
odd prime, writing $H^* (-)$ for singular cohomology with 
$\field_p$-coefficients. Here the difference stems from the fact that the 
enveloping algebra $UM$ of 
an unstable module $M$ is the quotient of the free symmetric algebra $S^* (M)$ 
by the relation $x^p = P_0 x$ for $x$ of even degree; in particular, the 
relation depends only upon the even degree elements of $M$.

The inclusion of the full subcategory $\unst ' \subset \unst$ of modules 
concentrated in even degree admits a right adjoint and induces an 
equivalence $\unst ' / \nil ' \cong \unst / \nil$, where $\nil' = \unst ' \cap 
\nil$ (see \cite[Section 5.1]{schwartz_book}). Since the final part of the 
proof 
works modulo nilpotents, this allows the avoidance of problems at $p$ 
odd associated with the action of the Bockstein (cf. \cite{Schwartz} and 
\cite{BHRS}). 

The fundamental exact sequence of functors (cf. Section 
\ref{subsect:basic_functors}) is now 
\[
 0 \rightarrow \id \stackrel{x \mapsto x^p}{\longrightarrow} S^p 
\longrightarrow 
\overline{S^p} \rightarrow 0,
\]
where the truncated symmetric power $\overline{S^p}$ is simple and self dual, 
together with the norm sequence:
\begin{eqnarray*}
 \xymatrix{
 0
 \ar[r]
 &
 \id 
 \ar[r]
 &
 S^p
 \ar[r]^N
 &
 \Gamma^p 
 \ar[r]
 &
 \id 
 \ar[r]
 &
 0
 }
\end{eqnarray*}
which represents a non-zero class in $\ext^2_{\f} (\id, \id)$ \cite{FLS}.

The $E^1$-page of the Arone-Goodwillie spectral sequence for  $X \mapsto 
\Sigma^\infty 
\Omega^n X$ at $p$ odd is calculated as in \cite{BHRS}. The arguments again use  
almost unstable modules (for the odd primary case); 
although only a very limited part of the structure of the $E^1$-page intervenes, 
it is necessary to establish that 
the spectral sequence is good almost unstable (as in Section 
\ref{sect:cohomEP}). 

One considers the $p$th filtration $F_p H^* (\Omega^n X) \subset H^* (\Omega^n 
X)$; the multiplicative structure induces a morphism of unstable modules 
\[
 \bigoplus _{j=2}^{p-1} S^j (F_1 H^*(\Omega^n X) )
 \rightarrow 
 F_p (H^* (\Omega^n X))
\]
which is a monomorphism in $\unst / \nil$. The cokernel 
$\overline{F_p} (H^* (\Omega^n X)) $ lies in  a short exact sequence in $\unst 
/\nil$:
\[
 0
 \rightarrow
 \Sigma^{-1} E^{-1, *}_\infty 
 \rightarrow 
  \overline{F_p} (H^* (\Omega^n X))
  \rightarrow 
   \Sigma^{-p} E^{-p, *}_\infty 
\rightarrow 
0.
\]
Moreover the cup product induces 
\[
 S^p (F_1 H^*(\Omega^n X))
 \rightarrow 
 \overline{F_p} (H^* (\Omega^n X))
\]
which is a monomorphism in $\unst / \nil$. 

The identification of $ \Sigma^{-p} E^{-p, *}_\infty $ modulo nilpotents relies 
on the calculation of $d_{p-1} : \Sigma   ( \Sigma^{-p} E^{-p, *}_{p-1}) 
\rightarrow   \Sigma^{-1} E^{-1, *}_{p-1}$, since \cite[Proposition 4.1]{BHRS} 
 shows that lower differentials act trivially on the dual Dyer-Lashof 
operations. 

Hereafter, the argument proceeds as for the case $p=2$, {\em mutatis mutandis}, 
by analysing $  
\overline{F_p} (H^* (\Omega^n X))$ (for $p=2$, this is simply $ F_2 (H^* 
(\Omega^n X))$). 

Since the proof reduces to an argument in $\unst/\nil$, the ability to work in 
$\unst'$ provides a useful simplification, making the 
parallel with the case $p=2$ transparent. For example,  for $n>1$, the analogue 
of Corollary \ref{cor:d1_niln} for $d_{p-1}$ reduces to considering  a morphism 
\[
 \Gamma^p (\rho'_n H^* (X)) \rightarrow \rho'_{n+1} H^* (X). 
\]
where, for $M$ an unstable module, $\rho'_k M$ denotes the largest submodule of 
$\rho_k M$ concentrated in even degrees. Moreover, Lemma 
\ref{lem:U_ses_Frobenius} has the following analogue: 
for $M' \in \obj \unst'$ there is a short exact sequence:
\[
 0 
 \rightarrow 
 \overline{S^p} M' 
 \rightarrow 
 \Gamma^p M'
 \rightarrow 
 \Phi M' 
 \rightarrow 
0
 \]
where the Frobenius functor $\Phi$ restricted to $\unst'$ is directly analogous 
to that for $p=2$.
As in the case $p=2$, the contribution $ \Phi(\rho'_n H^* (X))$ is provided by  
$\ddl_0$.

The proofs of the analogues of the results of Section \ref{sect:main} follow an 
identical
strategy, depending upon Theorem \ref{thm:split_mono_gen}, which is prime 
independent.

\begin{rem}
A refinement of the results can be obtained by exploiting the weight splitting 
of $\unst / \nil $  associated  
to the action of the multiplicative group $\field_p^\times$. 
\end{rem}


\begin{thebibliography}{BMMS86}

\bibitem[AK02]{AK}
Stephen~T. Ahearn and Nicholas~J. Kuhn, \emph{Product and other fine structure
  in polynomial resolutions of mapping spaces}, Algebr. Geom. Topol. \textbf{2}
  (2002), 591--647. \MR{1917068 (2003j:55009)}

\bibitem[BHRS13]{BHRS}
Sebastian B{\"u}scher, Fabian Hebestreit, Oliver R{\"o}ndigs, and Manfred
  Stelzer, \emph{The {A}rone-{G}oodwillie spectral sequence for
  {$\Sigma^\infty\Omega^n$} and topological realization at odd primes}, Algebr.
  Geom. Topol. \textbf{13} (2013), no.~1, 127--169. \MR{3031639}

\bibitem[BMMS86]{BMMS}
R.~R. Bruner, J.~P. May, J.~E. McClure, and M.~Steinberger, \emph{{$H_\infty $}
  ring spectra and their applications}, Lecture Notes in Mathematics, vol.
  1176, Springer-Verlag, Berlin, 1986. \MR{836132 (88e:55001)}

\bibitem[CGPS14]{CGPS}
N.~T. {Cuong}, G.~{Gaudens}, G.~{Powell}, and L.~{Schwartz}, \emph{{On
  non-realization results and conjectures of N. Kuhn}}, ArXiv e-prints,
  1402.2617[math.AT] (2014), To appear {\em Fundamenta Mathematicae}.

\bibitem[FLS94]{FLS}
Vincent Franjou, Jean Lannes, and Lionel Schwartz, \emph{Autour de la
  cohomologie de {M}ac {L}ane des corps finis}, Invent. Math. \textbf{115}
  (1994), no.~3, 513--538. \MR{1262942 (95d:19002)}

\bibitem[GLZ89]{GLZ}
J.~H. Gunawardena, J.~Lannes, and S.~Zarati, \emph{Cohomologie des groupes
  sym\'etriques et application de {Q}uillen}, Advances in homotopy theory
  ({C}ortona, 1988), London Math. Soc. Lecture Note Ser., vol. 139, Cambridge
  Univ. Press, Cambridge, 1989, pp.~61--68. \MR{1055868 (91d:18013)}

\bibitem[GS12]{GS12}
G{\'e}rald Gaudens and Lionel Schwartz, \emph{Realising unstable modules as the
  cohomology of spaces and mapping spaces}, Acta Math. Vietnam. \textbf{37}
  (2012), no.~4, 563--577. \MR{3058663}

\bibitem[HLS95]{HLS2}
Hans-Werner Henn, Jean Lannes, and Lionel Schwartz, \emph{Localizations of
  unstable {$A$}-modules and equivariant mod {$p$} cohomology}, Math. Ann.
  \textbf{301} (1995), no.~1, 23--68. \MR{1312569 (95k:55036)}

\bibitem[KM13]{KMcC}
Nicholas Kuhn and Jason McCarty, \emph{The mod 2 homology of infinite
  loopspaces}, Algebr. Geom. Topol. \textbf{13} (2013), no.~2, 687--745.
  \MR{3044591}

\bibitem[Kuh94]{KI}
Nicholas~J. Kuhn, \emph{Generic representations of the finite general linear
  groups and the {S}teenrod algebra. {I}}, Amer. J. Math. \textbf{116} (1994),
  no.~2, 327--360. \MR{1269607 (95c:55022)}

\bibitem[Kuh95a]{KIII}
\bysame, \emph{Generic representations of the finite general linear groups and
  the {S}teenrod algebra. {III}}, $K$-Theory \textbf{9} (1995), no.~3,
  273--303. \MR{1344142 (97c:55026)}

\bibitem[Kuh95b]{Kuhn_annals}
\bysame, \emph{On topologically realizing modules over the {S}teenrod algebra},
  Ann. of Math. (2) \textbf{141} (1995), no.~2, 321--347. \MR{1324137
  (96i:55027)}

\bibitem[Kuh08]{Kuhn_nonrealization}
Nicholas Kuhn, \emph{Topological nonrealization results via the {G}oodwillie
  tower approach to iterated loopspace homology}, Algebr. Geom. Topol.
  \textbf{8} (2008), no.~4, 2109--2129. \MR{2460881 (2010d:55024)}

\bibitem[Kuh14]{K14}
Nicholas~J. Kuhn, \emph{The {K}rull filtration of the category of unstable
  modules over the {S}teenrod algebra}, Math. Z. \textbf{277} (2014), no.~3-4,
  917--936. \MR{3229972}

\bibitem[LZ87]{LZ}
Jean Lannes and Sa{\"{\i}}d Zarati, \emph{Sur les foncteurs d\'eriv\'es de la
  d\'estabilisation}, Math. Z. \textbf{194} (1987), no.~1, 25--59. \MR{MR871217
  (88j:55014)}

\bibitem[Mil74]{Milgram}
R.~James Milgram, \emph{Unstable homotopy from the stable point of view},
  Lecture Notes in Mathematics, Vol. 368, Springer-Verlag, Berlin-New York,
  1974. \MR{0348740 (50 \#1235)}

\bibitem[Pow14]{p_destab}
Geoffrey M.~L. Powell, \emph{On the derived functors of destabilization at odd
  primes}, Acta Math. Vietnam. \textbf{39} (2014), no.~2, 205--236.
  \MR{3212661}

\bibitem[Pow15]{p_viasm}
\bysame, \emph{{On the derived functors of destabilization and of iterated loop
  functors}}, ArXiv:1503.08620 (2015).

\bibitem[PS14]{2014arXiv1408.3694P}
A.~{Putman} and S.~V {Sam}, \emph{{Representation stability and finite linear
  groups}}, ArXiv:1408.3694 (2014).

\bibitem[Sch94]{schwartz_book}
Lionel Schwartz, \emph{Unstable modules over the {S}teenrod algebra and
  {S}ullivan's fixed point set conjecture}, Chicago Lectures in Mathematics,
  University of Chicago Press, Chicago, IL, 1994. \MR{MR1282727 (95d:55017)}

\bibitem[Sch98]{Schwartz}
\bysame, \emph{\`{A} propos de la conjecture de non-r\'ealisation due \`a {N}.
  {K}uhn}, Invent. Math. \textbf{134} (1998), no.~1, 211--227. \MR{1646599
  (99j:55019)}

\bibitem[SS14]{2014arXiv1409.1670S}
S.~V {Sam} and A.~{Snowden}, \emph{{Gröbner methods for representations of
  combinatorial categories}}, ArXiv:1409.1670 (2014).

\end{thebibliography}
\providecommand{\bysame}{\leavevmode\hbox to3em{\hrulefill}\thinspace}
\providecommand{\MR}{\relax\ifhmode\unskip\space\fi MR }
\providecommand{\MRhref}[2]{%
  \href{http://www.ams.org/mathscinet-getitem?mr=#1}{#2}
}
\providecommand{\href}[2]{#2}

\end{document}